\documentclass[10pt,letterpaper,twoside,reqno]{amsart}

\makeatletter{}
\usepackage{fixltx2e}                       
\usepackage[usenames,dvipsnames]{xcolor}
\usepackage{fancyhdr}
\usepackage{amsmath,amsfonts,amsbsy,amsgen,amscd,mathrsfs,amssymb,amsthm}
\usepackage{subfig}
\usepackage{url}

\usepackage{mathtools}

\mathtoolsset{showonlyrefs}

\usepackage[font=small,margin=0.25in,labelfont={sc},labelsep={colon}]{caption}

\usepackage{tikz}
\usepackage{microtype}
\usepackage{enumitem}

\definecolor{dark-gray}{gray}{0.3}
\definecolor{dkgray}{rgb}{.4,.4,.4}
\definecolor{dkblue}{rgb}{0,0,.5}
\definecolor{medblue}{rgb}{0,0,.75}
\definecolor{rust}{rgb}{0.5,0.1,0.1}

\usepackage[colorlinks=true]{hyperref}

\hypersetup{urlcolor=rust}
\hypersetup{citecolor=dkblue}
\hypersetup{linkcolor=dkblue}

\usepackage{setspace}

\usepackage{graphicx}
\usepackage{booktabs,longtable,tabu} 
\usepackage{multirow} 
\usepackage{float}

\usepackage[full]{textcomp}

\usepackage{fourier}
\usepackage{bm} 
\graphicspath{{figures/}}

\newtheorem{theorem}{Theorem}[section]
\newtheorem{lemma}[theorem]{Lemma}

\newtheorem{proposition}[theorem]{Proposition}
\newtheorem{fact}[theorem]{Fact}

\newtheorem{corollary}[theorem]{Corollary}
\newtheorem{claim}[theorem]{Claim}

\theoremstyle{definition}

\newtheorem{definition}[theorem]{Definition}

\newtheorem{remark}[theorem]{Remark}

\newcommand{\term}{\emph}

\numberwithin{equation}{section} 
\numberwithin{figure}{section}
\numberwithin{table}{section}

\floatstyle{plaintop}
\newfloat{recipe}{thp}{lor}
\floatname{recipe}{Recipe}
\numberwithin{recipe}{section}

\providecommand{\mathbold}[1]{\bm{#1}}

\renewcommand{\phi}{\varphi}

\newcommand{\eps}{\varepsilon}

\newcommand{\half}{\tfrac{1}{2}}

\newcommand{\econst}{\mathrm{e}}
\newcommand{\iunit}{\mathrm{i}}

\newcommand{\Id}{\mathbf{I}}

\newcommand{\coll}[1]{\mathscr{#1}}

\providecommand{\mathbbm}{\mathbb} 
\newcommand{\R}{\mathbbm{R}}
\newcommand{\C}{\mathbbm{C}}

\newcommand{\abs}[1]{\left\vert {#1} \right\vert}
\newcommand{\abssq}[1]{{\abs{#1}}^2}

\newcommand{\real}{\operatorname{Re}}
\newcommand{\imag}{\operatorname{Im}}

\newcommand{\diff}[1]{\mathrm{d}{#1}}
\newcommand{\idiff}[1]{\, \diff{#1}}

\newcommand{\Expect}{\operatorname{\mathbb{E}}}

\newcommand{\vct}[1]{\mathbold{#1}}
\newcommand{\mtx}[1]{\mathbold{#1}}

\newcommand{\transp}{\mathsf{t}}
\newcommand{\adj}{*}

\newcommand{\trace}{\operatorname{tr}}
\newcommand{\ntr}{\operatorname{\bar{\trace}}}

\newcommand{\ip}[2]{\left\langle {#1},\ {#2} \right\rangle}

\newcommand{\abssqip}[2]{\abssq{\ip{#1}{#2}}}

\newcommand{\norm}[1]{\left\Vert {#1} \right\Vert}

\newcommand{\smnorm}[2]{{\bigl\Vert {#2} \bigr\Vert}_{#1}}

\newcommand{\pnorm}[2]{\norm{#2}_{#1}}

\newcommand{\triplenorm}[1]{\left\vert\!\left\vert\!\left\vert {#1} \right\vert\!\right\vert\!\right\vert}

\newcommand{\Cat}{\mathrm{Cat}}

\title{Second-Order Matrix Concentration Inequalities}

\author[J.~A.~Tropp]{Joel~A.~Tropp}
\thanks{Email: \url{jtropp@cms.caltech.edu}. Tel: 626.395.5957.}

\date{13 March 2015. Revised 21 April 2015 and 3 August 2016.}

\subjclass[2010]{Primary: 60B20. Secondary: 60F10, 60G50, 60G42.}
\keywords{Concentration inequality, moment inequality, random matrix.}

\begin{document}

\begin{abstract}
Matrix concentration inequalities give bounds for the spectral-norm
deviation of a random matrix from its expected value.  These results have a
weak dimensional dependence that is sometimes, but not always, necessary.
This paper identifies one of the sources of the dimensional term and exploits
this insight to develop sharper matrix concentration inequalities.
In particular, this analysis delivers two refinements of the matrix
Khintchine inequality that use information beyond the matrix variance
to reduce or eliminate the dimensional dependence.  
\end{abstract}

\maketitle

\section{Motivation}

Matrix concentration inequalities provide spectral information about
a random matrix that depends smoothly on many independent random variables.
In recent years, these results have become a dominant tool in applied
random matrix theory.  There are several reasons for the success of this approach.

\vspace{0.5pc}

\begin{itemize} \setlength{\itemsep}{0.5pc}
\item	\textbf{Flexibility.}  Matrix concentration applies to a wide range of random matrix models.
In particular, we can obtain bounds for the spectral norm of a sum of independent random matrices in terms
of the properties of the summands.

\item	\textbf{Ease of Use.}  For many applications, matrix concentration tools require only a small
amount of matrix analysis.  No expertise in random matrix theory is required to invoke the results.
 
\item	\textbf{Power.}  For a large class of examples, including independent sums,
matrix concentration bounds are provably close to optimal.
\end{itemize}

\vspace{0.5pc}

\noindent
See the monograph~\cite{Tro15:Introduction-Matrix} for an overview of this theory
and a comprehensive bibliography.

The matrix concentration inequalities in the literature are suboptimal
for certain examples because of a weak dependence on the dimension of the random matrix.
Removing this dimensional term is difficult because there are many situations where it is
necessary.
The purpose of this paper is to identify one of the sources of the dimensional factor.
Using this insight, we will develop some new matrix concentration
inequalities that are qualitatively better than the current generation of results,
although they sacrifice some of our desiderata.
Ultimately, we hope that this line of research will lead to general tools for
applied random matrix theory that are flexible, easy to use, and that give sharp
results in most cases.

\section{The Matrix Khintchine Inequality}

To set the stage, we present and discuss the primordial matrix concentration result,
the \term{matrix Khintchine inequality},
which describes the behavior of a special random matrix model, called a \term{matrix Gaussian series}.
This result already exhibits the key features of more sophisticated matrix concentration inequalities,
and it can be used to derive concentration bounds for more general models.
As such, the matrix Khintchine inequality serves as a natural starting point for deeper investigations.

\subsection{Matrix Gaussian Series}

In this work, we focus on an important class of random matrices
that has a lot of modeling power but still supports an interesting theory.

\begin{definition}[Matrix Gaussian Series]
Consider fixed Hermitian matrices $\mtx{H}_1, \dots, \mtx{H}_n$ with common dimension $d$,
and let $\{ \gamma_1, \dots, \gamma_n \}$ be an independent family of standard normal random variables.
Construct the random matrix
\begin{equation} \label{eqn:matrix-gauss-series}
\mtx{X} := \sum\nolimits_{i=1}^n \gamma_i \mtx{H}_i.
\end{equation}
We refer to a random matrix with this form as a matrix Gaussian series with Hermitian coefficients
or, for brevity, an \term{Hermitian matrix Gaussian series}.
\end{definition}

Matrix Gaussian series enjoy a surprising amount of modeling power.  It is easy to see that
we can express any random Hermitian matrix with jointly Gaussian entries in the
form~\eqref{eqn:matrix-gauss-series}.
More generally, we can use matrix Gaussian series to analyze a sum of independent,
zero-mean, random, Hermitian matrices $\mtx{Y}_1, \dots, \mtx{Y}_n$.  Indeed,
for any norm $\triplenorm{\cdot}$ on matrices,
\begin{equation} \label{eqn:cond-symm}
\Expect \triplenorm{ \sum\nolimits_{i=1}^n \mtx{Y}_i }
	\leq \sqrt{2\pi} \cdot \Expect\bigg[ \Expect\bigg[ \triplenorm{ \sum\nolimits_{i=1}^n \gamma_i \mtx{Y}_i }
	\, \big\vert\, \mtx{Y}_1, \dots, \mtx{Y}_n \bigg] \bigg].
\end{equation}
The process of passing from an independent sum to a conditional Gaussian series is called
\term{symmetrization}.  See~\cite[Lem.~6.3 and Eqn.~(4.8)]{LT91:Probability-Banach} for details
about this calculation.  Furthermore, some techniques for Gaussian series can be adapted to
study independent sums directly without the artifice of symmetrization.

Note that our restriction to Hermitian matrices is not really a limitation.  We can also
analyze a rectangular matrix $\mtx{Z}$ with jointly Gaussian entries by working
with the Hermitian dilation of $\mtx{Z}$, sometimes known as the Jordan--Wielandt matrix.
See~\cite[Sec.~2.1.16]{Tro15:Introduction-Matrix} for more information on this approach.

\subsection{The Matrix Variance}

Many matrix concentration inequalities are expressed most naturally in terms of a matrix
extension of the variance.

\begin{definition}[Matrix Variance]
Let $\mtx{X}$ be a random Hermitian matrix.  The \term{matrix variance} is the deterministic matrix
$$
\mathbf{Var}(\mtx{X}) := \Expect \mtx{X}^2 - (\Expect \mtx{X})^2.
$$
We use the convention that the power binds before the expectation.
\end{definition}

In particular, consider a matrix Gaussian series $\mtx{X} := \sum\nolimits_{i=1}^n \gamma_i \mtx{H}_i$.
It is easy to verify that $$
\mathbf{Var}(\mtx{X}) = \Expect \mtx{X}^2
	=  \sum\nolimits_{i,j=1}^n \Expect[ \gamma_i \gamma_j ] \cdot \mtx{H}_i \mtx{H}_j
	= \sum\nolimits_{i=1}^n \mtx{H}_i^2.
$$
We see that the matrix variance has a clean expression in terms of the coefficients of the
Gaussian series, so it is easy to compute in practice.

\subsection{The Matrix Khintchine Inequality}

The matrix Khintchine inequality is a fundamental fact about the behavior
of matrix Gaussian series.  The first version of this
result was established by Lust-Piquard~\cite{LP86:Inegalites-Khintchine}, and
the constants were refined in the papers~\cite{Pis98:Non-commutative-Vector,Buc01:Operator-Khintchine}.
The version here is adapted from~\cite[Sec.~7.1]{MJCFT14:Matrix-Concentration}.

\begin{proposition}[Matrix Khintchine] \label{prop:khintchine}
Consider an Hermitian matrix Gaussian series $\mtx{X} := \sum\nolimits_{i=1}^n \gamma_i \mtx{H}_i$,
as in~\eqref{eqn:matrix-gauss-series}.  Introduce the matrix standard deviation parameter
\begin{equation} \label{eqn:mtx-stdev-intro}
\sigma_{q}(\mtx{X}) := \pnorm{q}{ \mathbf{Var}(\mtx{X})^{1/2} }
	= \pnorm{q}{ \left(\sum\nolimits_{i=1}^n \mtx{H}_i^2 \right)^{1/2} }
\quad\text{for $q \geq 1$.}
\end{equation}
Then, for each integer $p \geq 1$,
\begin{equation} \label{eqn:khintchine-intro}
\sigma_{2p}(\mtx{X})
	\quad\leq\quad \big( \Expect \pnorm{2p}{ \mtx{X} }^{2p} \big)^{1/(2p)}
	\quad\leq\quad \sqrt{2p-1} \cdot \sigma_{2p}(\mtx{X}).
\end{equation}
The symbol $\pnorm{q}{\cdot}$ denotes the Schatten $q$-norm.
\end{proposition}

\noindent
The lower bound in~\eqref{eqn:khintchine-intro} is simply Jensen's inequality.
Section~\ref{sec:khintchine} contains a short proof of the upper bound.

The matrix Khintchine inequality also yields an estimate for the spectral
norm of a matrix Gaussian series.  This type of result is often more useful
in practice.

\begin{corollary}[Matrix Khintchine: Spectral Norm] \label{cor:khintchine-spectral}
Consider an Hermitian matrix Gaussian series $\mtx{X} := \sum\nolimits_{i=1}^n \gamma_i \mtx{H}_i$ with dimension $d$,
as in~\eqref{eqn:matrix-gauss-series}.  Introduce the matrix standard deviation parameter
$$
\sigma(\mtx{X}) := \norm{ \mathbf{Var}(\mtx{X}) }^{1/2} = \norm{ \sum\nolimits_{i=1}^n \mtx{H}_i^2 }^{1/2}.
$$
Then
\begin{equation} \label{eqn:khintchine-spectral}
\frac{1}{\sqrt{2}} \cdot \sigma(\mtx{X})
	\quad\leq\quad \Expect \norm{ \mtx{X} }
	\quad\leq\quad \sqrt{\econst \, (1 + 2\log d)} \cdot \sigma(\mtx{X}).
\end{equation}
The symbol $\norm{\cdot}$ denotes the spectral norm, also known as the $\ell_2$ operator norm.
\end{corollary}

\begin{proof}[Proof Sketch]
For the upper bound, observe that
$$
\Expect \norm{\mtx{X}} 	\leq \left( \Expect \pnorm{2p}{ \mtx{X} }^{2p} \right)^{1/(2p)}
	\leq \sqrt{2p - 1} \cdot \pnorm{2p}{ \mathbf{Var}(\mtx{X})^{1/2} }
	\leq d^{1/(2p)} \sqrt{2p - 1} \cdot \norm{ \mathbf{Var}(\mtx{X}) }^{1/2}.
$$
Indeed, the spectral norm is bounded above by the Schatten $2p$-norm, and we can
apply Lyapunov's inequality to increase the order of the moment from one to $2p$.
Invoke Proposition~\ref{prop:khintchine}, and bound the trace in terms
of the spectral norm again.  Finally, set $p = \lceil \log d \rceil$, and simplify the constants.

For the lower bound, note that
$$
\Expect \norm{ \mtx{X} }
	\geq \frac{1}{\sqrt{2}} \left( \Expect \norm{ \sum\nolimits_{i=1}^n \gamma_i \mtx{H}_i }^2 \right)^{1/2}
	=\frac{1}{\sqrt{2}} \big( \Expect \norm{\smash{\mtx{X}^2}} \big)^{1/2}
	\geq \frac{1}{\sqrt{2}} \norm{ \mathbf{Var}(\mtx{X}) }^{1/2}.
$$
The first relation follows from the optimal Khintchine--Kahane
inequality~\cite{LO94:Best-Constant}; the last is Jensen's.
\end{proof}

\subsection{Two Examples}
\label{sec:intro-examples}

The bound~\eqref{eqn:khintchine-spectral} shows that
the matrix standard deviation controls the expected norm of a matrix Gaussian series up
to a factor that is logarithmic in the dimension of the random matrix.  One may
wonder whether the lower branch or the upper branch of~\eqref{eqn:khintchine-spectral} gives
the more accurate result.  In fact, natural examples demonstrate that both extremes of behavior occur.

For an integer $d \geq 1$, define
\begin{equation} \label{eqn:intro-diag}
\mtx{X}_{\rm diag} :=
\mtx{X}_{\rm diag}(d) := \begin{bmatrix} \gamma_1 \\ & \gamma_2 \\ && \gamma_3 \\ &&& \ddots \\ &&&& \gamma_d \end{bmatrix}.
\end{equation}
That is, $\mtx{X}_{\rm diag}$ is a $d \times d$ diagonal matrix whose entries
$\{ \gamma_i : 1 \leq i \leq d \}$ are independent standard normal variables.
Second, define
\begin{equation} \label{eqn:intro-goe}
\mtx{X}_{\rm goe} := \mtx{X}_{\rm goe}(d) := \frac{1}{\sqrt{2d}} (\mtx{G} + \mtx{G}^\adj)
\quad\text{where}\quad
\mtx{G} := \mtx{G}(d) := \begin{bmatrix}
	\gamma_{11} & \gamma_{12} &\dots & \gamma_{1d} \\
	\gamma_{21} & \gamma_{22} & \dots & \gamma_{2d} \\
	\vdots & \vdots & \ddots & \vdots \\
	\gamma_{d1} & \gamma_{d2} & \dots & \gamma_{dd}
\end{bmatrix}
\end{equation}
The symbol ${}^*$ denotes conjugate transposition.
Up to scaling, the $d \times d$ random matrix $\mtx{X}_{\rm goe}$ is the Hermitian part
of a matrix $\mtx{G}$ whose entries $\{ \gamma_{ij} : 1 \leq i, j \leq d \}$
are independent standard normal variables.
The sequence $\{ \mtx{X}_{\rm goe}(d) : d = 1, 2, 3, \dots \} $
is called the \term{Gaussian orthogonal ensemble} (GOE).

To apply the matrix Khintchine inequality, we represent each matrix as an
Hermitian Gaussian series:
$$
\mtx{X}_{\rm diag} = \sum\nolimits_{i=1}^d \gamma_i \, \mathbf{E}_{ii}
\quad\text{and}\quad
\mtx{X}_{\rm goe} = \frac{1}{\sqrt{2d}}
	\sum\nolimits_{1 \leq i, j \leq d} \gamma_{ij} \, (\mathbf{E}_{ij} + \mathbf{E}_{ji}).
$$
We have written $\mathbf{E}_{ij}$ for the $d \times d$ matrix with a one
in the $(i, j)$ position and zeros elsewhere.
Respectively, the matrix variances satisfy
$$
\mathbf{Var}(\mtx{X}_{\rm diag}) = \Id
\quad\text{and}\quad
\mathbf{Var}(\mtx{X}_{\rm goe}) = \big(1 + d^{-1} \big) \cdot \Id.
$$
The bound~\eqref{eqn:khintchine-spectral} delivers
\begin{equation} \label{eqn:khintchine-examples}
\frac{1}{\sqrt{2}} \quad\leq\quad \Expect \norm{\mtx{X}} \quad\lessapprox\quad \sqrt{2 \econst \log d}
\qquad\text{for $\mtx{X} = \mtx{X}_{\rm diag}$ or $\mtx{X} = \mtx{X}_{\rm goe}$.}
\end{equation}
The relations $\lessapprox$ and $\approx$ suppress lower-order terms.
In each case, the ratio between the lower and upper bound has order $\sqrt{\log d}$.
The matrix Khintchine inequality does not provide more precise information.

On the other hand, for these examples, detailed spectral information is available:
\begin{equation} \label{eqn:khintchine-examples-2}
\Expect \norm{ \smash{\mtx{X}_{\rm goe}} } \approx 2
\quad\text{and}\quad
\Expect \norm{ \smash{\mtx{X}_{\rm diag}} } \approx \sqrt{2 \log d}.
\end{equation}
See~\cite[Sec.~2.3]{Tao12:Topics-Random} for a proof of the result on the GOE matrix;
the bound for the diagonal matrix depends on the familiar calculation of the
expected maximum of $d$ independent standard normal random variables.  We see that
the norm of the GOE matrix is close to the lower bound provided
by~\eqref{eqn:khintchine-spectral}, while the norm of the diagonal
matrix is close to the upper bound.

\subsection{A Question}

Corollary~\ref{cor:khintchine-spectral} shows that the matrix variance controls the expected norm
of a matrix Gaussian series.  On the other hand, the two examples in
the previous section demonstrate that we need more information than
the variance to determine the norm up to a constant factor.  Therefore, we must ask...

\begin{quotation} \bf
Are there parameters that allow us to calculate the norm of a matrix
Gaussian series more precisely than the matrix variance?
\end{quotation}

\noindent
This paper provides the first affirmative answer to this question.

\section{Beyond the Matrix Khintchine Inequality}

This section presents new results that improve on the matrix Khintchine inequality,
Proposition~\ref{prop:khintchine}.  First, we motivate the type of parameters that
arise when we try to refine this result.  Then we define a quantity,
called the \term{matrix alignment parameter}, that describes how the coefficients
in the matrix Gaussian series interact with each other.
In Section~\ref{sec:k2-intro}, we use the alignment parameter to state a new bound that provides
a uniform improvement over the matrix Khintchine inequality.
Further refinements are possible if we consider random matrices with highly symmetric
distributions, so we introduce the class of strongly isotropic random matrices
in Section~\ref{sec:strong-isotrope}.  Section~\ref{sec:k2-strong-intro} contains a matrix
Khintchine inequality for matrix Gaussian series that are strongly isotropic.  This
bound is good enough to compute the norm of a large GOE matrix exactly.
Finally, in Sections~\ref{sec:discussion} and~\ref{sec:related}, we discuss extensions and related work.

\subsection{Prospects} \label{sec:prospects}

What kind of parameters might allow us to refine Proposition~\ref{prop:khintchine}?
The result is already an identity for $p = 1$.
For inspiration, let us work out what happens when $p = 2$:
$$
\begin{aligned}
\Expect \pnorm{4}{ \mtx{X}  }^4
	= \Expect \trace \left( \sum\nolimits_{i=1}^n \gamma_i \mtx{H}_i \right)^4
	&= \sum\nolimits_{i,j,k,\ell=1}^n \Expect[ \gamma_i \gamma_j \gamma_k \gamma_\ell ] \cdot
	\trace \big[ \mtx{H}_i \mtx{H}_j \mtx{H}_k \mtx{H}_{\ell} \big] \\
	&= 2 \trace \left( \sum\nolimits_{i,j=1}^n \mtx{H}_i^2 \right)^2 
	+ \trace \left( \sum\nolimits_{i,j=1}^n \mtx{H}_i \mtx{H}_j \mtx{H}_i \mtx{H}_j \right)
	=: 2 \trace \mathbf{Var}(\mtx{X})^2 + \trace \mtx{\Delta}.
\end{aligned}
$$
We use the convention that powers bind before the trace.
The product of Gaussian variables has expectation zero unless the indices are
paired.  In the last expression, the first term comes from the cases where $i = j$ and $k = \ell$ or where
$i = \ell$ and $j = k$; the second term comes from the case where $i = k$ and $j = \ell$.
Once again, the matrix variance $\mathbf{Var}(\mtx{X})$ emerges,
but we have a new second-order term $\mtx{\Delta}$
that arises from the summands where the indices alternate: ($i,j,i,j$).\footnote{A related observation animates the theory of free probability, which gives
a fine description of certain large random matrices. The key fact about centered, free
random variables $Y$ and $Z$ is that crossing moments, such as $\phi( YZYZ )$,
must vanish~\cite{NS06:Lectures-Combinatorics}.}

In a sense, the matrix $\mtx{\Delta}$ reflects the extent to which the coefficient matrices are aligned.
When the family $\{\mtx{H}_i\}$ commutes, the matrix $\mtx{\Delta} = \mathbf{Var}(\mtx{X})^2$,
so the second-order term provides no new information.  More generally, whenever the coefficients commute,
the quantity $(\Expect \pnorm{2p}{\mtx{X}}^{2p})^{1/(2p)}$ can be expressed in terms of the matrix variance
and the number $p$, and the matrix Khintchine inequality, Proposition~\ref{prop:khintchine},
gives an estimate of the correct order.  In other words, commuting coefficients
are the worst possible circumstance.  Most previous work on matrix concentration
implicitly uses this worst-case model in the analysis.

To achieve better results, we need to account for how the coefficient matrices $\mtx{H}_i$
interact with each other.  The calculation above suggests that the matrix $\mtx{\Delta}$
might contain the information we need.  Heuristically, when the coefficients fail to
commute, the matrix $\mtx{\Delta}$ should be small.  As we will see, this idea is fruitful,
but we need a parameter more discerning than $\mtx{\Delta}$.

Let us summarize this discussion in the following observation:

\begin{quotation} \bf
To improve on the matrix Khintchine inequality, we must quantify
the extent to which the coefficient matrices commute.
\end{quotation}

\noindent
Our work builds on this intuition to establish new matrix concentration inequalities.

\subsection{The Matrix Alignment Parameter}
\label{sec:matrix-alignment}

In this section, we introduce a new parameter for a matrix Gaussian series that
describes how much the coefficients commute with each other.  In later sections, we will
present extensions of the matrix Khintchine inequality that rely on this parameter.

\begin{definition}[Matrix Alignment Parameter]
Let $\mtx{H}_1, \dots, \mtx{H}_n$ be Hermitian matrices with dimension $d$.
For each $p \geq 1$, the \term{matrix alignment parameter} of this sequence is the quantity
\begin{equation} \label{eqn:def-matrix-alignment}
w_{p} := \max_{\mtx{Q}_{\ell}} \ \pnorm{p}{ \abs{
	\sum\nolimits_{i,j=1}^n \mtx{H}_i \mtx{Q}_1 \mtx{H}_j \mtx{Q}_2 \mtx{H}_i \mtx{Q}_3 \mtx{H}_j }^{1/4} }
	\quad\text{and}\quad
w := w_{\infty}.
\end{equation}
The maximum takes place over a triple $(\mtx{Q}_1, \mtx{Q}_2, \mtx{Q}_3)$ of unitary matrices
with dimension $d$.  The matrix absolute value is defined as $\abs{\mtx{B}} := (\mtx{B}^* \mtx{B})^{1/2}$.
\end{definition}

\noindent
Roughly, the matrix alignment parameter~\eqref{eqn:def-matrix-alignment}
describes how well the matrices $\mtx{H}_1, \dots, \mtx{H}_n$
can be aligned with each other under the worst choices of coordinates.

The quantity~\eqref{eqn:def-matrix-alignment} appears mysterious,
so it is worth a few paragraphs to clarify its meaning.  First, let us
compare the alignment parameter with the matrix standard deviation
parameter~\eqref{eqn:mtx-stdev-intro} that appears in the matrix Khintchine
inequality.

\begin{proposition}[Standard Deviation versus Alignment] \label{prop:sigma-versus-w}
Let $\mtx{H}_1, \dots, \mtx{H}_n$ be Hermitian matrices.  Define the standard deviation
and alignment parameters
$$
\sigma_{p} := \pnorm{p}{ \left( \sum\nolimits_{i=1}^n \mtx{H}_i \right)^{1/2} }
\quad\text{and}\quad
w_{p} := \max_{ \mtx{Q}_{\ell} }\ \pnorm{p}{ \abs{
	\sum\nolimits_{i,j=1}^n \mtx{H}_i \mtx{Q}_1 \mtx{H}_j \mtx{Q}_2 \mtx{H}_i \mtx{Q}_3 \mtx{H}_j }^{1/4} }
	\quad\text{for $p \geq 1$.}
$$
Then
$$
w_{p} \leq \sigma_{p}
\quad\text{for all $p \geq 4$.}
$$
\end{proposition}

\noindent
The proof of Proposition~\ref{prop:sigma-versus-w} appears in Section~\ref{sec:interleave-stdev}.

Next, let us return to the examples in the introduction.
In Section~\ref{sec:compute-interleave}, we provide detailed calculations of the 
standard deviation and alignment parameters.
For the diagonal Gaussian series $\mtx{X}_{\rm diag}$ defined in~\eqref{eqn:intro-diag}, we have
$$
\sigma(\mtx{X}_{\rm diag}) = 1
\quad\text{and}\quad
w(\mtx{X}_{\rm diag}) = 1.
$$
For the GOE matrix $\mtx{X}_{\rm goe}$ defined in~\eqref{eqn:intro-goe},
\begin{equation} \label{eqn:goe-stats-intro}
\sigma(\mtx{X}_{\rm goe}) = 1 + d^{-1}
\quad\text{and}\quad
w(\mtx{X}_{\rm goe}) \leq (4d)^{-1/4}.
\end{equation}
The matrix alignment parameter can tell the two examples apart, while
the matrix standard deviation cannot!

\begin{remark}[Notation for Alignment]
Here and elsewhere, we abuse notation by writing $w_p(\mtx{X})$ and $w(\mtx{X})$ for the alignment parameter
of a matrix Gaussian series $\mtx{X} := \sum\nolimits_{i=1}^n \gamma_i \mtx{H}_i$,
even though $w$ is a function of the coefficient matrices $\mtx{H}_i$ in the representation
of the series.
\end{remark}

\begin{remark}[Are the Unitaries Necessary?]
At this stage, it may seem capricious to include the unitary matrices
in the definition~\eqref{eqn:def-matrix-alignment}.  In fact, the example in
Section~\ref{sec:spin} demonstrates that the alignment parameter would
lose its value if we were to remove the unitary matrices.  On the other
hand, there are situations where the unitary matrices are not completely
arbitrary, as discussed in Section~\ref{sec:k2-discussion}.
\end{remark}

\subsection{A Second-Order Matrix Khintchine Inequality}
\label{sec:k2-intro}

The first major result of this paper is an improvement on the matrix Khintchine inequality.
This theorem uses the second-order information in the alignment parameter
to obtain better bounds. 
\begin{theorem}[Second-Order Matrix Khintchine] \label{thm:k2-intro}
Consider an Hermitian matrix Gaussian series $\mtx{X} := \sum\nolimits_{i=1}^n \gamma_i \mtx{H}_i$,
as in~\eqref{eqn:matrix-gauss-series}.
Define the matrix standard deviation and matrix alignment parameters
$$
\sigma_{q}(\mtx{X}) := \pnorm{q}{ \left( \sum\nolimits_{i=1}^n \mtx{H}_i^2 \right)^{1/2} }
\quad\text{and}\quad
w_{q}(\mtx{X}) := \max_{\mtx{Q}_{\ell}} \ \pnorm{q}{ \abs{ \sum\nolimits_{i,j=1}^n
	\mtx{H}_i \mtx{Q}_1 \mtx{H}_j \mtx{Q}_2 \mtx{H}_i \mtx{Q}_3 \mtx{H}_j }^{1/4} }
	\quad\text{for $q \geq 1$.}
$$
The maximum takes place over a triple $(\mtx{Q}_1, \mtx{Q}_2, \mtx{Q}_3)$ of unitary matrices.
Then, for each integer $p \geq 3$,
\begin{equation} \label{eqn:k2-intro}
\big( \Expect \pnorm{2p}{\mtx{X}}^{2p} \big)^{1/(2p)}
	\quad\leq\quad 3 \sqrt[4]{2p - 5} \cdot \sigma_{2p}(\mtx{X})
	\ +\ \sqrt{2p - 4} \cdot w_{2p}(\mtx{X}).
\end{equation}
The symbol $\pnorm{q}{\cdot}$ denotes the Schatten $q$-norm.
\end{theorem}

\noindent
The proof of Theorem~\ref{thm:k2-intro} appears in Section~\ref{sec:k2}.

We can also derive bounds for the spectral norm of a matrix
Gaussian series.

\begin{corollary}[Second-Order Matrix Khintchine: Spectral Norm]
Consider an Hermitian matrix Gaussian series $\mtx{X} := \sum\nolimits_{i=1}^n \gamma_i \mtx{H}_i$
with dimension $d \geq 8$, as in~\eqref{eqn:matrix-gauss-series}.
Define the matrix standard deviation and matrix alignment parameters
$$
\sigma(\mtx{X}) := \norm{ \sum\nolimits_{i=1}^n \mtx{H}_i^2 }^{1/2}
\quad\text{and}\quad
w(\mtx{X}) := \max_{\mtx{Q}_{\ell}} \ \norm{ \sum\nolimits_{i,j=1}^n
	\mtx{H}_i \mtx{Q}_1 \mtx{H}_j \mtx{Q}_2 \mtx{H}_i \mtx{Q}_3 \mtx{H}_j }^{1/4}.
$$
The maximum ranges over a triple $(\mtx{Q}_1, \mtx{Q}_2, \mtx{Q}_3)$ of unitary matrices.
Then
\begin{equation} \label{eqn:k2-spectral-bd}
\Expect \norm{\mtx{X}}
	\quad\leq\quad 3\sigma(\mtx{X}) \sqrt[4]{2\econst \log d}
	\ +\ w(\mtx{X}) \sqrt{2\econst \log d}.
\end{equation}
The symbol $\norm{\cdot}$ denotes the spectral norm.
\end{corollary}

\noindent
The result follows from Theorem~\ref{thm:k2-intro} by setting $p = \lceil \log d \rceil$.
The potential gain in~\eqref{eqn:k2-spectral-bd} over~\eqref{eqn:khintchine-spectral}
comes from the reduction of the power on the first logarithm
from one-half to one-quarter.

\subsection{Matrix Khintchine versus Second-Order Matrix Khintchine}

Let us make some comparisons between Proposition~\ref{prop:khintchine} and
Theorem~\ref{thm:k2-intro}.  First, recall that the alignment parameter is
dominated by the standard deviation parameter: $w_q(\mtx{X}) \leq \sigma_q(\mtx{X})$ for $q \geq 4$
because of Proposition~\ref{prop:sigma-versus-w}.
Therefore, the bound~\eqref{eqn:k2-intro} implies that
$$
\big( \Expect \pnorm{2p}{\mtx{X}}^{2p} \big)^{1/(2p)}
	\leq \big(3 \sqrt[4]{2p-1} + \sqrt{2p-1} \big) \cdot \sigma_{2p}(\mtx{X})
	\quad\text{for $p = 3, 4, 5, \dots$.}
$$
This is very close to the prediction from Proposition~\ref{prop:khintchine},
so Theorem~\ref{thm:k2-intro} is never significantly worse.

On the other hand, there are situations where Theorem~\ref{thm:k2-intro}
gives qualitatively better results.  In particular, for the GOE matrix
$\mtx{X}_{\rm goe}(d)$, the bound~\eqref{eqn:k2-intro} and the
calculation~\eqref{eqn:goe-stats-intro} yield
$$
\Expect \norm{ \smash{\mtx{X}_{\rm goe}(d)} }
	\leq 3 \sqrt[4]{2 \econst \log d} + \frac{\sqrt{2 \econst \log d}}{\sqrt[4]{4d}}
	\quad\text{for $d \geq 8$.}
$$
This estimate beats our first attempt in~\eqref{eqn:khintchine-examples},
but it still falls short of the correct estimate $\Expect \norm{\smash{\mtx{X}_{\rm goe}}} \approx 2$.

\subsection{Strongly Isotropic Random Matrices}
\label{sec:strong-isotrope}

As we have seen, Theorem~\ref{thm:k2-intro} offers a qualitative improvement
over the matrix Khintchine inequality, Proposition~\ref{prop:khintchine}.
Nevertheless, the new result still lacks the power to determine the norm
of the GOE matrix correctly.  We can obtain more satisfactory results by specializing
our attention to a class of random matrices with highly symmetric distributions.

\begin{definition}[Strong Isotropy]
Let $\mtx{X}$ be a random Hermitian matrix.
We say that $\mtx{X}$ is \term{strongly isotropic}
when
$$
\Expect \mtx{X}^p
= \big( \Expect \ntr \mtx{X}^p \big) \cdot \Id
\quad\text{for $p = 0, 1, 2, \dots$.}
$$
The symbol $\ntr$ denotes the normalized trace: $\ntr \mtx{A} := d^{-1} \trace \mtx{A}$
when $\mtx{A}$ has dimension $d$.
\end{definition}

The easiest way to check that a random matrix is strongly isotropic
is to exploit symmetry properties of the distribution.  We offer one
of many possible results in this direction~\cite[Lem.~7.1]{CT14:Subadditivity-Matrix}.

\begin{proposition}[Strong Isotropy: Sufficient Condition] \label{prop:signed-perm}
Let $\mtx{X}$ be a random Hermitian matrix.
Suppose that the distribution of $\mtx{X}$ is invariant under signed permutation:
$$
\mtx{X} \sim \mtx{\Pi}^\adj \mtx{X} \mtx{\Pi}
\quad\text{for every signed permutation $\mtx{\Pi}$.}
$$
Then $\mtx{X}$ is strongly isotropic.
The symbol $\sim$ refers to equality of distribution.
A signed permutation is a square matrix that has precisely one nonzero entry in each row
or column, this entry taking the values $\pm 1$.
\end{proposition}

\begin{proof}
Suppose that $\mtx{\Pi}$ is a signed permutation, drawn uniformly
at random.  For $p = 0, 1, 2, 3, \dots$,
$$
\Expect \mtx{X}^p = \Expect\big[ (\mtx{\Pi}^{\adj} \mtx{X} \mtx{\Pi})^p \big]
	= \Expect\big[ \Expect\big[ \mtx{\Pi}^\adj \mtx{X}^p \mtx{\Pi} \, \big\vert \, \mtx{X} \big] \big]
	= \Expect\big[ \big( \ntr \mtx{X}^p \big) \cdot \Id \big]
	= \big(\Expect \ntr \mtx{X}^p \big) \cdot \Id
$$
The first relation uses invariance under signed permutation, and
the second relies on the fact that signed permutations are unitary.
Averaging a fixed matrix over signed permutations yields
the identity times the normalized trace of the matrix.
\end{proof}

Proposition~\ref{prop:signed-perm} applies to many types of random matrices.
In particular, the diagonal Gaussian matrix $\mtx{X}_{\rm diag}$
and the GOE matrix $\mtx{X}_{\rm goe}$ are both strongly isotropic because of this result.
Other types of distributional symmetry can also lead to strong isotropy.

\begin{remark}[Group Orbits]
Here is a more general class of matrix Gaussian series where we can
verify strong isotropy using abstract arguments.
Let $\coll{G}$ be a unitary representation of a finite group, and
let $\mtx{A}$ be a fixed Hermitian matrix with the same dimension.
Consider the random Hermitian matrix
$$
\mtx{X} := \sum_{\mtx{U} \in \coll{G}} \gamma_{\mtx{U}} \cdot \mtx{U} \mtx{A} \mtx{U}^\adj
$$
where $\{ \gamma_{\mtx{U}} : \mtx{U} \in \coll{G} \}$ is an independent family of
standard normal variables.
Since $\coll{G}$ acts on itself by permutation,
$$
\mtx{UXU}^\adj \sim \mtx{X}
\quad\text{for each $\mtx{U} \in \coll{G}$.}
$$
This observation allows us to perform averaging arguments like the one
in Proposition~\ref{prop:signed-perm}.

There are several ways to apply this property to
argue that $\mtx{X}$ is strongly isotropic.
For example, it suffices that
$$
\coll{G}' := \{ \mtx{M} : \mtx{MU} = \mtx{UM} \text{ for all $\mtx{U} \in \coll{G}$ } \}
	= \{ z \Id : z \in \C \}.
$$
It is also sufficient that $\{ \mtx{U} \vct{a} : \mtx{U} \in \coll{G} \}$
forms a (complete) tight frame for every vector $\vct{a}$;
see the paper~\cite{VW08:Tight-Frames} for some
situations where this condition holds.
\end{remark}

\begin{remark}[Spherical Designs]
A \term{spherical $t$-design} is a collection $\{ \vct{u}_i : i = 1, \dots, N \}$
of points on the unit sphere $\mathbb{S}^{d-1}$ in $\R^d$ with the property that
$$
\int_{\mathbb{S}^{d-1}} \varphi(\vct{u}) \idiff{\mu}(\vct{u})
= \frac{1}{N} \sum_{i=1}^N \varphi(\vct{u}_i)
$$
where $\varphi$ is an arbitrary algebraic polynomial in $d$ variables with degree $t$
and $\diff{\mu}$ is the Haar measure on the sphere $\mathbb{S}^{d-1}$.
See the paper~\cite{BRV13:Optimal-Asymptotic} for existence
results and background references.

Given a spherical $t$-design, consider the random matrix
$$
\mtx{X} := \sum_{i=1}^N \gamma_i \vct{u}_i \vct{u}_i^\adj.
$$
where $\{\gamma_i : i = 1, \dots, N \}$ is an independent family
of standard normal variables.
By construction, this random matrix has the property that
$$
\Expect \mtx{X}^{p} = (\Expect \ntr \mtx{X}^p) \cdot \Id
\quad\text{for $p = 0, 1, 2, \dots, \lfloor t/2 \rfloor$.}
$$
This variant of the strong isotropy property is sufficient
for many purposes, provided that $t \approx \log d$.
\end{remark}

\subsection{A Second-Order Khintchine Inequality under Strong Isotropy}
\label{sec:k2-strong-intro}

The second major result of this paper is a second-order matrix Khintchine
inequality that is valid for matrix Gaussian series with the strong isotropy property.
Like Theorem~\ref{thm:k2-intro}, this result uses the alignment parameter
to control the norm of the random matrix.

\begin{theorem}[Second-Order Khintchine under Strong Isotropy] \label{thm:k2-strong-intro}
Consider an Hermitian matrix Gaussian series $\mtx{X} := \sum\nolimits_{i=1}^n \gamma_i \mtx{H}_i$
with dimension $d$, as in~\eqref{eqn:matrix-gauss-series}, and assume that $\mtx{X}$ is strongly isotropic.
Introduce the matrix standard deviation and matrix alignment parameters:
$$
\sigma(\mtx{X}) := \norm{ \sum\nolimits_{i=1}^n \mtx{H}_i^2  }^{1/2}
\quad\text{and}\quad
w(\mtx{X}) := \max_{\mtx{Q}_{\ell}} \ \norm{ \sum\nolimits_{i,j=1}^n
	\mtx{H}_i \mtx{Q}_1 \mtx{H}_j \mtx{Q}_2 \mtx{H}_i \mtx{Q}_3 \mtx{H}_j }^{1/4}.
$$
The maximum ranges over a triple $(\mtx{Q}_1, \mtx{Q}_2, \mtx{Q}_3)$ of unitary matrices.
Then, for each integer $p \geq 1$,
$$
\big( \Expect \pnorm{2p}{\mtx{X}}^{2p} \big)^{1/(2p)}
	\quad\leq\quad \left[ 2 \sigma(\mtx{X}) \ + \ 2^{1/4} \, p^{5/4} w(\mtx{X}) \right] \cdot d^{1/(2p)}.
$$
The symbol $\norm{\cdot}$ refers to the spectral norm, while $\pnorm{q}{\cdot}$ is the Schatten $q$-norm.
\end{theorem}

\noindent
The proof of this result appears in Section~\ref{sec:k2-strong},
where we also establish a lower bound.

Theorem~\ref{thm:k2-strong-intro} shows that the  moments of the random matrix $\mtx{X}$ are controlled
by the standard deviation $\sigma(\mtx{X})$ whenever $p^{5/4} w(\mtx{X}) \ll \sigma(\mtx{X})$.
If we take $p = \lceil \log d \rceil$,
the Schatten $2p$-norm is essentially the same as the spectral norm,
and the dimensional factor on the right-hand side is negligible.
Therefore,
$$
w(\mtx{X}) \log^{5/4} d \ll \sigma(\mtx{X})
\quad\text{implies}\quad
\Expect \norm{\mtx{X}} \lessapprox 2 \sigma(\mtx{X}).
$$
In the presence of strong isotropy, the spectral norm of a matrix Gaussian series is
comparable with the standard deviation $\sigma(\mtx{X})$ whenever the alignment parameter
$w(\mtx{X})$ is relatively small!

In particular, we can apply this result to the GOE matrix $\mtx{X}_{\rm goe}$
because of Proposition~\ref{prop:signed-perm}.
The calculation~\eqref{eqn:goe-stats-intro} of the standard deviation
and alignment parameters ensures that
$
\Expect \norm{\smash{\mtx{X}_{\rm goe}}} \lessapprox 2.
$
As we observed in~\eqref{eqn:khintchine-examples-2}, this bound is sharp.
For this example, we can even take $p \approx d^{1/5}$,
which leads to very good probability bounds via Markov's inequality.
Furthermore, a more detailed version of Theorem~\ref{thm:k2-strong-intro},
appearing in Section~\ref{sec:k2-strong}, is precise enough to show
that the semicircle law is the limiting spectral distribution of the GOE.

On the other hand, the dependence on the exponent $p$ in Theorem~\ref{thm:k2-strong-intro}
is suboptimal.
This point is evident when we consider the diagonal Gaussian matrix $\mtx{X}_{\rm diag}(d)$.
Indeed, Theorem~\ref{thm:k2-strong-intro} only implies the bound
$$
\Expect \norm{\smash{\mtx{X}_{\rm diag}}} \leq {\rm const} \cdot \log^{5/4} d.
$$
As we observed in~\eqref{eqn:khintchine-examples-2}, the power on the logarithm should be one-half.

\subsection{Discussion}
\label{sec:discussion}

This paper opens a new chapter in the theory of matrix concentration
and noncommutative moment inequalities.
Our main technical contribution is to demonstrate that
the matrix Khintchine inequality, Proposition~\ref{prop:khintchine},
is not the last word on the behavior of a matrix Gaussian series.
Indeed, we have shown that the matrix variance does not contain sufficient
information to determine the expected norm of a matrix Gaussian series.
We have also identified another quantity, the matrix alignment parameter,
that allows us to obtain better bounds for every matrix Gaussian series.
Furthermore, in the presence of more extensive distributional information,
it is even possible to obtain numerically sharp bounds for the norm of certain
matrix Gaussian series.

There are a number of ways to extend the ideas and results in this paper:

\begin{description} \setlength{\itemsep}{0.5pc}

\item	[Higher-Order Alignment]  If we consider alignment parameters involving
$2k$ coefficient matrices, it is possible to improve the term $p^{1/4} \sigma_{2p}$
in Theorem~\ref{thm:k2-intro} to $p^{1/(2k)} \sigma_{2p}$.  See Section~\ref{sec:k2-discussion}
for some additional details.

\item	[Other Matrix Series]  We can use exchangeable pairs techniques~\cite{MJCFT14:Matrix-Concentration}
to study matrix series of the form $\mtx{X} := \sum\nolimits_{i=1}^n \xi_i \mtx{H}_i$ where $\{ \xi_i \}$
is an independent family of scalar random variables.  This approach is potentially quite interesting when
the $\xi_i$ are Bernoulli (that is, 0--1) random variables.

\item	[Independent Sums]  We can use conditioning and symmetrization, as in~\eqref{eqn:cond-symm},
to apply Theorem~\ref{thm:k2-intro} to a sum of independent random matrices.
See~\cite[App.]{CGT12:Masked-Sample} for an example of this type of argument.

\item	[Rectangular Matrices]  The techniques here also give results for rectangular random matrices by way of
the Hermitian dilation~\cite[Sec.~2.1.16]{Tro15:Introduction-Matrix}.  In this setting, a different notion
of strong isotropy becomes relevant; see Section~\ref{sec:k2-strong-discussion}.

\end{description}

\noindent
We have not elaborated on these ideas because there is
also evidence that alignment parameters will not lead to final
results on matrix concentration.  

\subsection{Related Work}
\label{sec:related}

There are very few techniques in the literature on random matrices that
satisfy all three of our three requirements: flexibility, ease of use, and power.
In particular, for many practical applications, it is important to be able to work with
an arbitrary sum of independent random matrices.  We have chosen to study matrix
Gaussian series because they are the simplest instance of this model,
and they may lead to further insights about the general problem.

Most classical work in random matrix theory concerns very special classes
of random matrices; the books~\cite{BS10:Spectral-Analysis,Tao12:Topics-Random}
provide an overview of some of the main lines of research in this field.
There are some specific subareas of random matrix theory that address more
general models.  The monograph~\cite{NS06:Lectures-Combinatorics} gives an introduction to
free probability.  The book chapter~\cite{Ver12:Introduction-Nonasymptotic}
describes a collection of methods from Banach space geometry.
The monograph~\cite{Tro15:Introduction-Matrix} covers the theory of
matrix concentration.  The last three works have a wide scope of applicability,
but none of them provides the ultimate description of the behavior of a sum of
independent random matrices.

There is one specific strand of research that we would like to draw out because it
is very close in spirit to this paper.
Recently, Bandeira \& van Handel~\cite{BV14:Sharp-Nonasymptotic}
and van Handel~\cite{VH15:Spectral-Norm} have studied the behavior
of a real symmetric Gaussian matrix whose entries are
independent and centered but have inhomogeneous variances (the \term{independent-entry model}).
A $d \times d$ random matrix from this class can be written as
$$
\mtx{X}_{\rm indep} := \sum\nolimits_{i,j=1}^d a_{ij} \gamma_{ij} \cdot (\mathbf{E}_{ij} + \mathbf{E}_{ji})
\quad\text{for $a_{ij} \in \R$.}
$$
As usual, $\{\gamma_{ij}\}$ is an independent family of standard normal random variables,
and we assume that $a_{ij} = a_{ji}$ without loss of generality.

To situate this model in the context of our work, observe that matrix Gaussian series
are significantly more general than the independent-entry model.
The strongly isotropic model is incomparable with the independent-entry model.
To see why, recall that strongly isotropic matrices can have dependent entries.
At the same time, $\Expect \mtx{X}_{\rm indep}^p$ is diagonal for each integer $p \geq 0$,
but it need not be a scalar matrix.

For the independent-entry model, Bandeira \& van Handel~\cite{BV14:Sharp-Nonasymptotic}
established the following (sharp) bound:
\begin{equation} \label{eqn:bvh}
\Expect \norm{\smash{\mtx{X}_{\rm indep}}}
	\lessapprox 2\sigma(\mtx{X}_{\rm indep}) + {\rm const} \cdot \max\nolimits_{ij} \abs{\smash{a_{ij}}}
	\cdot \sqrt{\log d}
\end{equation}
The maximum entry $\max_{ij} \abs{\smash{a_{ij}}}$ plays the same role in this formula as the
matrix alignment parameter plays in this paper.
The paper~\cite{BV14:Sharp-Nonasymptotic} leans heavily on the independence assumption,
so it is not clear whether the ideas extend to a more general setting.

To compare the result~\eqref{eqn:bvh} with the work here, we can compute the matrix alignment parameter
for the independent-entry model using a difficult extension of the calculation in Section~\ref{sec:gauss-wigner}.
This effort yields
$$
w(\mtx{X}_{\rm indep}) \approx \left( \max\nolimits_i \sum\nolimits_{j} \abs{\smash{a_{ij}}}^4 \right)^{1/4}.
$$
We see that the matrix alignment parameter is somewhat larger than the maximum entry $\max_{ij} \abs{\smash{a_{ij}}}$.
Thus, for the independent model, Theorem~\ref{thm:k2-intro} gives us a better result than the classical
Khintchine inequality, Proposition~\ref{prop:khintchine}, but it is somewhat weaker than~\eqref{eqn:bvh}.
Theorem~\ref{thm:k2-strong-intro} would give a result close to the bound~\eqref{eqn:bvh},
but it does not always apply because the independent-entry model need not be strongly isotropic.

The independent-entry model is not adequate to reach results with the
same power and scope as the current generation of matrix concentration
bounds~\cite{Tro15:Introduction-Matrix}.  Nevertheless, the estimate~\eqref{eqn:bvh}
strongly suggests that there are better ways of summarizing the interactions of the
coefficients in an Hermitian matrix Gaussian series $\mtx{X} := \sum\nolimits_{i=1}^n \gamma_{i} \mtx{H}_i$
than the alignment parameter $w(\mtx{X})$.  One possibility is the weak variance parameter:
$$
\sigma_{\star}(\mtx{X}) := \sup_{\norm{\vct{u}}=\norm{\vct{v}}=1}
	\left( \sum\nolimits_{i=1}^n \abssqip{\vct{u}}{\mtx{H}_i \vct{v}} \right)^{1/2}.
$$
For the independent-entry  model, this quantity reduces to ${\rm const} \cdot \max_{ij} \abs{\smash{a_{ij}}}$.
The idea of considering $\sigma_{\star}(\mtx{X})$ is motivated by the discussion
in~\cite[Sec.~4]{Tro12:User-Friendly-FOCM}, as well as the work in~\cite{BV14:Sharp-Nonasymptotic,VH15:Spectral-Norm}.
Unfortunately, at this stage, it is not clear whether there are any parameters that allow us to
obtain a simple description of the behavior of a Gaussian matrix in the absence of burdensome
independence or isotropy assumptions.  This is a frontier for future work.

\section{Computation of the Matrix Alignment Parameters}
\label{sec:compute-interleave}

In this section, we show how to compute the matrix alignment parameter for the
two random matrices in the introduction, the diagonal Gaussian matrix and the GOE matrix.
Afterward, we show by example that neither Theorem~\ref{thm:k2-intro} nor Theorem~\ref{thm:k2-strong-intro}
can hold if we remove the unitary factors from the matrix alignment parameter.

\subsection{A Diagonal Gaussian Matrix}

The diagonal Gaussian matrix takes the form
$$
\mtx{X}_{\rm diag} := \sum\nolimits_{i=1}^d \gamma_i \mathbf{E}_{ii}.
$$
The matrix variance $\mathbf{Var}(\mtx{X}_{\rm diag}) = \Expect \mtx{X}_{\rm diag}^2 = \Id$.
It follows that the matrix standard deviation parameters, defined in~\eqref{eqn:mtx-stdev-intro},
satisfy
$$
\sigma_{p}(\mtx{X}_{\rm diag})
	= \pnorm{p}{ \mathbf{Var}(\mtx{X}_{\rm diag})^{1/2} }
	= d^{1/p}
\quad\text{for $1 \leq p \leq \infty$}.
$$
We will show that the matrix alignment parameters, defined in~\eqref{eqn:def-matrix-alignment},
satisfy
$$
w_{p}(\mtx{X}_{\rm diag}) = d^{1/p}
\quad\text{for $4 \leq p \leq \infty$}.
$$
Thus, for this example, the second-order matrix Khintchine inequalities, Theorem~\ref{thm:k2-intro}
and Theorem~\ref{thm:k2-strong-intro}, do not improve over the matrix Khintchine inequality,
Proposition~\ref{prop:khintchine}.
This outcome is natural, given that the classical result is essentially optimal in this case.

Let us evaluate the matrix alignment parameter.
For a triple $(\mtx{Q}, \mtx{S}, \mtx{U})$ of unitary matrices, form the sum
$$
\mtx{W}(\mtx{Q}, \mtx{S}, \mtx{U}) :=
\sum\nolimits_{i,j=1}^d \mathbf{E}_{ii} \mtx{Q} \mathbf{E}_{jj} \mtx{S} \mathbf{E}_{ii} \mtx{U} \mathbf{E}_{jj}
	= \sum\nolimits_{i,j=1}^d q_{ij} s_{ji} u_{ij} \cdot \mathbf{E}_{ij}
	= \mtx{Q} \odot \mtx{S}^\transp \odot \mtx{U}.
$$
We have written $\odot$ for the Schur (i.e., componentwise) product, and ${}^\transp$ is the
transpose operation.  When $\mtx{Q} = \mtx{S} = \mtx{U} = \Id$, the sum collapses:
$\mtx{W}(\Id, \Id, \Id) = \Id$.  Therefore,
$$
w_{p}(\mtx{X}_{\rm diag}) = \max_{\mtx{Q},\mtx{S},\mtx{U}} \ \pnorm{p}{ \abs{\mtx{W}(\mtx{Q}, \mtx{S}, \mtx{U})}^{1/4} }
	\geq \pnorm{p}{ \Id } = d^{1/p}
	\quad\text{for $p \geq 1$.}
$$
But Proposition~\ref{prop:sigma-versus-w} shows that
$$
w_{p}(\mtx{X}_{\rm diag}) \leq \sigma_{p}(\mtx{X}_{\rm diag}) = d^{1/p}
\quad\text{for each $p \geq 4$.}
$$
Therefore, $w_p(\mtx{X}_{\rm diag}) = \sigma_p(\mtx{X}_{\rm diag}) = d^{1/p}$ for $p \geq 4$.
The result for $p = \infty$ follows when we take limits.

\begin{remark}[Commutativity]
A similar calculation is valid whenever the family $\{\mtx{H}_i\}$
of coefficient matrices in the matrix Gaussian series~\eqref{eqn:matrix-gauss-series}
commutes.
\end{remark}

\subsection{A GOE Matrix}
\label{sec:gauss-wigner}

The GOE matrix takes the form
$$
\mtx{X}_{\rm goe} := \frac{1}{\sqrt{2d}} \sum\nolimits_{i,j=1}^d \gamma_{ij} (\mathbf{E}_{ij} + \mathbf{E}_{ji}).
$$
An easy calculation shows that the matrix variance satisfies
$$
\mathbf{Var}(\mtx{X}_{\rm goe}) = \Expect \mtx{X}_{\rm goe}^2
	= \frac{1}{2d} \sum\nolimits_{i,j=1}^d (\mathbf{E}_{ij} + \mathbf{E}_{ji})^2
	= (1 + d^{-1}) \cdot \Id.
$$
Therefore, the matrix standard deviation parameters, defined in~\eqref{eqn:mtx-stdev-intro},
equal
$$
\sigma_{p}(\mtx{X}_{\rm goe}) 	= \pnorm{p}{ \mathbf{Var}(\mtx{X}_{\rm goe})^{1/2} }
	= \sqrt{1+d^{-1}} \cdot d^{1/p}
	\quad\text{for $1 \leq p \leq \infty$.}
$$
We will demonstrate that the matrix alignment parameters, defined in~\eqref{eqn:def-matrix-alignment},
satisfy
$$
w_{p}(\mtx{X}_{\rm goe}) \leq \big( d^{-1} + 3 d^{-2} \big)^{1/4} \cdot d^{1/p}
\quad\text{for $4 \leq p \leq \infty$}.
$$
When $d$ is large, the matrix alignment parameters are much smaller than the matrix standard deviation
parameters.  As a consequence, the second-order matrix Khintchine inequalities deliver a substantial gain over
the classical matrix Khintchine inequality.

Let us compute the matrix alignment parameter. For a triple $(\mtx{Q}, \mtx{S}, \mtx{U})$ of unitary matrices,
introduce the (unnormalized) sum
$$
\mtx{W}(\mtx{Q}, \mtx{S}, \mtx{U}) :=
	\sum\nolimits_{i_1, i_2, j_1, j_2 = 1}^d
		(\mathbf{E}_{i_1 i_2} + \mathbf{E}_{i_2 i_1}) \mtx{Q}
		(\mathbf{E}_{j_1 j_2} + \mathbf{E}_{j_2 j_1}) \mtx{S}
		(\mathbf{E}_{i_1 i_2} + \mathbf{E}_{i_2 i_1}) \mtx{U}
		(\mathbf{E}_{j_1 j_2} + \mathbf{E}_{j_2 j_1}).
$$
It is not hard to evaluate this sum if we take care.  First, distribute terms:
$$
\begin{aligned}
\mtx{W}(\mtx{Q}, \mtx{S}, \mtx{U}) =
\sum\nolimits_{i_1, i_2, j_1,j_2 = 1}^{d}
	\bigg[ &\big( q_{i_2 j_1} s_{j_2 i_1} u_{i_2 j_2}
	   + q_{i_2 j_1} s_{j_2 i_2} u_{i_1 j_2}
	   + q_{i_2 j_2} s_{j_1 i_1} u_{i_2 j_2}
	   + q_{i_2 j_2} s_{j_1 i_2} u_{i_1 j_2} \big) \cdot \mathbf{E}_{i_1 j_1} \phantom{\bigg]}\\
	& + \big( q_{i_2 j_1} s_{j_2 i_1} u_{i_2 j_1}
	   + q_{i_2 j_1} s_{j_2 i_2} u_{i_1 j_1}
	   + q_{i_2 j_2} s_{j_1 i_1} u_{i_2 j_1}
	   + q_{i_2 j_2} s_{j_1 i_2} u_{i_1 j_1} \big) \cdot \mathbf{E}_{i_1 j_2} \phantom{\bigg]}\\
	& +\big( q_{i_1 j_1} s_{j_2 i_1} u_{i_2 j_2}
	   + q_{i_1 j_1} s_{j_2 i_2} u_{i_1 j_2}
	   + q_{i_1 j_2} s_{j_1 i_1} u_{i_2 j_2}
	   + q_{i_1 j_2} s_{j_1 i_2} u_{i_1 j_2} \big) \cdot \mathbf{E}_{i_2 j_1} \phantom{\bigg]}\\
	& +\big(q_{i_1 j_1} s_{j_2 i_1} u_{i_2 j_1}
	   + q_{i_1 j_1} s_{j_2 i_2} u_{i_1 j_1}
	   + q_{i_1 j_2} s_{j_1 i_1} u_{i_2 j_1}
	   + q_{i_1 j_2} s_{j_1 i_2} u_{i_1 j_1} \big) \cdot \mathbf{E}_{i_2 j_2} \bigg].
\end{aligned}
$$
In each line, we can sum through the two free indices to identify four matrix products.
For example, in the first line, we can sum on $i_2$ and $j_2$.  This step yields
$$
\begin{aligned}
\mtx{W}(\mtx{Q}, \mtx{S}, \mtx{U}) =
&\sum\nolimits_{i_1, j_1 = 1}^{d}
	\big( \mtx{S}^\transp \mtx{U}^\transp \mtx{Q} 
	   + \mtx{USQ} 
	   + \trace(\mtx{Q}^\transp \mtx{U}) \cdot \mtx{S}^\transp 
	   + \mtx{U} \mtx{Q}^\transp \mtx{S}^\transp \big)_{i_1 j_1} \cdot \mathbf{E}_{i_1 j_1} \\
+ &\sum\nolimits_{i_1, j_2 = 1}^d
	\big( \trace(\mtx{Q}^\transp \mtx{U}) \cdot \mtx{S}^\transp 
	   + \mtx{U} \mtx{Q}^\transp \mtx{S}^\transp
	   + \mtx{S}^\transp \mtx{U}^\transp \mtx{Q}
	   + \mtx{USQ} \big)_{i_1 j_2} \cdot \mathbf{E}_{i_1 j_2} \\
+ &\sum\nolimits_{i_2, j_1 = 1}^d
	\big( \mtx{USQ}  
	   + \mtx{S}^\transp \mtx{U}^\transp \mtx{Q}  
	   + \mtx{U} \mtx{Q}^\transp \mtx{S}^\transp  
	   + \trace(\mtx{QU}) \cdot \mtx{S}^\transp \big)_{i_2 j_1} \cdot \mathbf{E}_{i_2 j_1}\\
+ &\sum\nolimits_{i_2, j_2 = 1}^d
	\big( \mtx{U}\mtx{Q}^\transp \mtx{S}^\transp
	   + \trace(\mtx{QU}) \cdot \mtx{S}^\transp
	   + \mtx{USQ}
	   + \mtx{S}^\transp \mtx{U}^\transp \mtx{Q} \big)_{i_2 j_2} \cdot \mathbf{E}_{i_2 j_2} \bigg].
\end{aligned}
$$
Sum through the remaining indices to reach
$$
\begin{aligned}
\mtx{W}(\mtx{Q}, \mtx{S}, \mtx{U}) = \
	&\big( \mtx{S}^\transp \mtx{U}^\transp \mtx{Q} 
	   + \mtx{USQ} 
	   + \trace(\mtx{Q}^\transp \mtx{U}) \cdot \mtx{S}^\transp 
	   + \mtx{U} \mtx{Q}^\transp \mtx{S}^\transp \big) \\
	& + \big( \trace(\mtx{Q}^\transp \mtx{U}) \cdot \mtx{S}^\transp 
	   + \mtx{U} \mtx{Q}^\transp \mtx{S}^\transp
	   + \mtx{S}^\transp \mtx{U}^\transp \mtx{Q}
	   + \mtx{USQ} \big) \\
	& +\big( \mtx{USQ}  
	   + \mtx{S}^\transp \mtx{U}^\transp \mtx{Q}  
	   + \mtx{U} \mtx{Q}^\transp \mtx{S}^\transp  
	   + \trace(\mtx{QU}) \cdot \mtx{S}^\transp \big) \\
	& +\big( \mtx{U}\mtx{Q}^\transp \mtx{S}^\transp
	   + \trace(\mtx{QU}) \cdot \mtx{S}^\transp
	   + \mtx{USQ}
	   + \mtx{S}^\transp \mtx{U}^\transp \mtx{Q} \big).
\end{aligned}
$$
Twelve of the sixteen terms are unitary matrices, and the remaining four
are scaled unitary matrices.  Furthermore, each trace is bounded in magnitude by $d$,
the worst case being $\mtx{Q} = \mtx{U} = \Id$.  Applying the definition of
the Schatten norm, the triangle inequality, and unitary invariance, we find that
$$
\pnorm{p}{ \abs{ \mtx{W}(\mtx{Q}, \mtx{S}, \mtx{U}) }^{1/4} }
	= \pnorm{p/4}{ \mtx{W}(\mtx{Q}, \mtx{S}, \mtx{U}) }^{1/4}
	\leq \left( (4d + 12) \cdot \pnorm{p/4}{ \Id } \right)^{1/4}
	= (4d + 12)^{1/4} \cdot d^{1/p}
	\quad\text{for $p \geq 4$.}
$$
To compute $w_p(\mtx{X}_{\rm goe})$, we must reintroduce the scaling $(2d)^{-1/2}$,
which gives the advertised result:
$$
w_p(\mtx{X}_{\rm goe}) \leq (2d)^{-1/2} \cdot (4d + 12)^{1/4} \cdot d^{1/p}
	= \big(d^{-1} + 3 d^{-2} \big)^{1/4} \cdot d^{1/p}.
$$
To obtain the bound for $p = \infty$, we simply take limits.

\subsection{The Unitaries are Necessary}
\label{sec:spin}

Suppose that $\mtx{X} := \sum\nolimits_{i=1}^n \gamma_i \mtx{H}_i$
is an Hermitian matrix Gaussian series with dimension $d$,
and let $\sigma(\mtx{X})$ be the matrix standard deviation~\eqref{eqn:mtx-stdev-intro}.
Consider the alternative alignment parameter
$$
\delta(\mtx{X}) := \norm{ \sum\nolimits_{i,j=1}^n \mtx{H}_i \mtx{H}_j \mtx{H}_i \mtx{H}_j }^{1/4}.
$$
This quantity is suggested by the discussion in Section~\ref{sec:prospects}.
Consider a general estimate of the form
\begin{equation} \label{eqn:putative-bd}
\Expect \norm{\mtx{X}}
	\quad\leq\quad f(d) \cdot \sigma(\mtx{X}) \ + \ g(d) \cdot \delta(\mtx{X}).
\end{equation}
We will demonstrate that, for every choice of the function $g$,
there is a lower bound $f(d) \geq {\rm const} \cdot \sqrt{\log d}$.
From this claim, we deduce that it is impossible to improve over the classical Khintchine
inequality by using the second-order quantity $\delta(\mtx{X})$.  Therefore, the unitary
matrices in the alignment parameter $w(\mtx{X})$ play a critical role.
Most of this argument was developed by Afonso Bandeira;
we are grateful to him for allowing us to include it.

Introduce the Pauli spin matrices
$$
\mtx{H}_1 := \begin{bmatrix} 1 & 0 \\ 0 & -1 \end{bmatrix} \qquad
\mtx{H}_2 := \begin{bmatrix} 0 & 1 \\ 1 & 0 \end{bmatrix} \qquad
\mtx{H}_3 := \begin{bmatrix} 0 & \iunit \\ -\iunit & 0 \end{bmatrix}.
$$
These matrices are Hermitian and unitary, so $\mtx{H}_i^2 = \Id$ for $i = 1, 2, 3$.
Furthermore, they satisfy the relations $(\mtx{H}_i \mtx{H}_j)^2 = - \Id$
when $i \neq j$.
Next, define $\mtx{H}_0 := \sqrt{\alpha}\, \Id$, where $\alpha := 2\sqrt{3} - 3$.  Calculate that
$$
\begin{aligned}
\sum\nolimits_{i,j=0}^3 \mtx{H}_i \mtx{H}_j \mtx{H}_i \mtx{H}_j
	&= \sum\nolimits_{i=0}^4 \mtx{H}_i^4 + \sum\nolimits_{j=1}^3 \mtx{H}_0 \mtx{H}_j \mtx{H}_0 \mtx{H}_j
	+ \sum\nolimits_{i=1}^3 \mtx{H}_i \mtx{H}_0 \mtx{H}_i \mtx{H}_0
	+ \sum\nolimits_{\substack{i, j = 1 \\ i \neq j}}^3 \mtx{H}_i \mtx{H}_j \mtx{H}_i \mtx{H}_j \\\
	&= (\alpha^2 + 3) \, \Id + 6 \alpha \, \Id - 6 \, \Id
	= (\alpha^2 +6\alpha - 3) \, \Id
	= \mtx{0}.
\end{aligned}
$$
Indeed, $\alpha$ is a positive root of the quadratic.

Consider the two-dimensional Gaussian series $\mtx{Y}$ generated by the matrices $\mtx{H}_0, \dots, \mtx{H}_3$:
$$
\mtx{Y} := \sum\nolimits_{i=0}^3 \gamma_i \mtx{H}_i.
$$
As usual, $\{\gamma_i\}$ is an independent family of standard normal variables.
For the series $\mtx{Y}$, we have already shown that the alternative alignment parameter $\delta(\mtx{Y}) = 0$.
Let us compute the variance and standard deviation:
$$
\mathbf{Var}(\mtx{Y}) = \sum\nolimits_{i=0}^3 \mtx{H}_i^2 = (\alpha + 3) \, \Id
	= 2 \sqrt{3} \, \Id
	\quad\text{and}\quad
	\sigma(\mtx{Y}) = \norm{ \mathbf{Var}(\mtx{Y}) }^{1/2}
	= 12^{1/4}.
$$
Expanding the random matrix $\mtx{Y}$ in coordinates, we also find that
$$
\mtx{Y} = \begin{bmatrix} \sqrt{\alpha} \, \gamma_0 + \gamma_1 & \gamma_2 + \iunit \gamma_3 \\
	\gamma_2 - \iunit \gamma_3 & \sqrt{\alpha} \, \gamma_0 - \gamma_1 \end{bmatrix}.
$$
Therefore, the top-left entry $(\mtx{Y})_{11}$ is a centered normal random variable
with variance $1 + \alpha = 2 (\sqrt{3} - 1)$.

To obtain the counterexample to the bound~\eqref{eqn:putative-bd}, fix an integer $d \geq 1$.
Let $\mtx{Y}_1, \dots, \mtx{Y}_d$ be independent copies of the two-dimensional
Gaussian series $\mtx{Y}$, and construct the $2d$-dimensional matrix Gaussian series
$$
\mtx{X}_{\rm spin} := \mtx{Y}_1 \oplus \dots \oplus \mtx{Y}_d
	= \sum\nolimits_{j=1}^{d} \mathbf{E}_{jj} \otimes \mtx{Y}_j
	\sim \sum\nolimits_{j=1}^{d} \sum\nolimits_{i=0}^3 \gamma_{ij} \, (\mathbf{E}_{jj} \otimes \mtx{H}_i).
$$
We have written $\oplus$ for direct sum and $\otimes$ for the Kronecker product; the matrices $\mathbf{E}_{jj}$
are the diagonal units with dimension $d \times d$; and $\{\gamma_{ij}\}$ is an independent family of standard
normal variables.

Extending the calculations above, we find that $\sigma(\mtx{X}_{\rm spin}) = 12^{1/4}$ and $\delta(\mtx{X}_{\rm spin}) = 0$.
Meanwhile, the norm of $\mtx{X}_{\rm spin}$ is bounded below by the absolute value of each of its diagonal entries.
In particular,
$$
\Expect \norm{ \smash{\mtx{X}_{\rm spin}} } \geq \Expect \max\nolimits_{j} \abs{ \smash{(\mtx{Y}_j)_{11}} }
	\geq {\rm const} \cdot \big( 2(\sqrt{3} - 1) \big)^{1/2} \cdot \sqrt{\log d}.
$$
We have used the fact that the expected maximum of $d$ independent standard normal variables
is proportional to $\sqrt{\log d}$.  Assuming that~\eqref{eqn:putative-bd} is valid,
we can sequence these estimates to obtain
$$
{\rm const} \cdot \sqrt{\log d}
	\leq f(d) \cdot \sigma(\mtx{X}_{\rm spin}) \ + \ g(d) \cdot \delta(\mtx{X}_{\rm spin})
	= 12^{1/4} \cdot f(d).
$$
Therefore, the function $f(d)$ must grow at least as fast as $\sqrt{\log d}$.
We conclude that a bound of the form~\eqref{eqn:putative-bd} can never improve
over the classical matrix Khintchine inequality.

\section{Notation \& Background}

Before we enter into the body of the paper, let us set some additional notation and state
a few background results.  First, $\mathbb{M}_d$ denotes the complex linear space of $d \times d$ matrices with complex entries.
We write $\mathbb{H}_d$ for the real-linear subspace of $\mathbb{M}_d$ that consists of Hermitian
matrices.  The symbol ${}^*$ represents conjugate transposition.
We write $\mtx{0}$ for the zero matrix and $\Id$ for the identity.
The matrix $\mathbf{E}_{ij}$ has a one in the $(i, j)$ position and zeros
elsewhere.  The dimensions of these matrices are typically determined by context.

For an Hermitian matrix $\mtx{A}$, we define the integer powers $\mtx{A}^{p}$
for $p = 0, 1, 2, 3, \dots$ in the usual way by iterated multiplication.
For a positive-semidefinite matrix $\mtx{P}$, we can also define
complex powers $\mtx{P}^z$
by raising each eigenvalue of $\mtx{P}$ to the power $z$ while
maintaining the eigenvectors.
In particular, $\mtx{P}^{1/2}$ is the unique positive-semidefinite square root of $\mtx{P}$.
The matrix absolute value is defined for a general matrix $\mtx{B}$
by the rule $\abs{\mtx{B}} := (\mtx{B}^\adj \mtx{B})^{1/2}$.  Note
that $\abs{\mtx{P}} = \mtx{P}$ when $\mtx{P}$ is positive semidefinite.

The \term{trace} and \term{normalized trace} of a matrix are given by
$$
\trace \mtx{B} := \sum\nolimits_{i=1}^d b_{ii}
\quad\text{and}\quad
\ntr \mtx{B} := \frac{1}{d} \sum\nolimits_{i=1}^d b_{ii}
\quad\text{for $\mtx{B} \in \mathbb{M}_d$.}
$$
We use the convention that a power binds before the trace to avoid unnecessary parentheses;
powers also bind before expectation.
The Schatten $p$-norm is defined for an arbitrary matrix $\mtx{B}$ via the rule
$$
\pnorm{p}{\mtx{B}} := \big( \trace \abs{\mtx{B}}^{p} \big)^{1/p}
\quad\text{for $p \geq 1$.}
$$
The Schatten $\infty$-norm $\pnorm{\infty}{\cdot}$ coincides with the spectral norm $\norm{\cdot}$.
This work uses both trace powers and Schatten norms, depending on which one is
conceptually clearer.
We require some H{\"o}lder inequalities involving the trace and the Schatten norms.
For matrices $\mtx{A}, \mtx{B} \in \mathbb{M}_d$ and $\varrho \geq 1$,
\begin{equation} \label{eqn:holder}
\abs{ \trace( \mtx{AB} ) }
	\leq \big( \trace \abs{\mtx{A}}^\varrho \big)^{1/\varrho} \cdot \big( \trace \abs{\mtx{B}}^{\varrho'} \big)^{1/\varrho'}
	\quad\text{where $\varrho' := \varrho / (\varrho - 1)$.}
\end{equation}
Furthermore,
\begin{equation} \label{eqn:holder-schatten}
\pnorm{\varrho}{ \smash{\mtx{A}^\adj\mtx{B}} }^2
	\leq \pnorm{\varrho}{ \smash{\mtx{A}^\adj \mtx{A}} }
	\cdot \pnorm{\varrho}{ \smash{\mtx{B}^\adj \mtx{B}} }.
\end{equation}
These results are drawn from~\cite[Chap.~IV]{Bha97:Matrix-Analysis}.

\section{The Trace Moments of a Matrix Gaussian Series}

For each major result in this paper, the starting point is a formula for the trace moments
of a matrix Gaussian series.

\begin{lemma}[Trace Moment Identity] \label{lem:trace-moments}
Let $\mtx{X} := \sum\nolimits_{i=1}^n \gamma_i \mtx{H}_i$ be an Hermitian matrix Gaussian series,
as in~\eqref{eqn:matrix-gauss-series}.
For each integer $p \geq 1$, we have the identity
\begin{equation} \label{eqn:gauss-trace-moments}
	\Expect \trace \mtx{X}^{2p}
	= \sum\nolimits_{q=0}^{2p-2} \sum\nolimits_{i=1}^n \Expect \trace\big[ \mtx{H}_i \mtx{X}^q \mtx{H}_i \mtx{X}^{2p - 2 - q} \big].
\end{equation}
\end{lemma}

\noindent
The easy proof of Lemma~\ref{lem:trace-moments} appears in the next two subsections.

Integration by parts is not foreign in the study of Gaussian random matrices;
for example, see~\cite[Sec.~2.4.1]{AGZ10:Introduction-Random} or~\cite[Sec.~9]{Kem13:Introduction-Random}.
The exchangeable pairs method for establishing matrix concentration is also based on an elementary,
but conceptually challenging, analog of integration by parts~\cite[Lem.~2.4]{MJCFT14:Matrix-Concentration}.
Aside from these works, we are not aware of any application of related techniques to prove results
on matrix concentration.

\subsection{Preliminaries}

To obtain Lemma~\ref{lem:trace-moments}, the main auxiliary tool
is the classical integration by parts formula for a function of a standard
normal vector~\cite[Lem.~1.1.1]{NP12:Normal-Approximations}.
In the form required here, the result can be derived with basic calculus.

\begin{fact}[Gaussian Integration by Parts] \label{fact:gauss-ibp}
Let $\vct{\gamma} \in \R^n$ be a vector with independent standard normal entries,
and let $f : \R^n \to \R$ be a function whose derivative is absolutely integrable with
respect to the standard normal measure.  Then
\begin{equation} \label{eqn:ibp2}
\sum\nolimits_{i=1}^n \Expect\left[ \gamma_i \cdot f(\vct{\gamma}) \right]
	= \sum\nolimits_{i=1}^n \Expect \left[ (\partial_i f)(\vct{\gamma}) \right].
\end{equation}
The symbol $\partial_i$ denotes differentiation with respect to the $i$th coordinate.
\end{fact}

We also use a well-known formula for the derivative of a matrix power~\cite[Sec.~X.4]{Bha97:Matrix-Analysis}.

\begin{fact}[Derivative of a Matrix Power] \label{fact:matrix-power}
Let $\mtx{A} : \R \to \mathbb{M}_d$ be a differentiable function.
For each integer $\varrho \geq 1$,
\begin{equation} \label{eqn:d-power}
\frac{\diff{}}{\diff{u}} \left( \mtx{A}(u)^{\varrho} \right)
	= \sum\nolimits_{k=0}^{\varrho-1} \mtx{A}(u)^\varrho \cdot \frac{\diff{}}{\diff{u}} \mtx{A}(u) \cdot \mtx{A}(u)^{\varrho-1-k}.
\end{equation}
In particular,
\begin{equation} \label{eqn:d-trace-power}
\frac{\diff{}}{\diff{u}} \trace \mtx{A}(u)^\varrho 
	= \varrho \cdot \trace\left[ \mtx{A}(u)^{\varrho-1} \cdot \frac{\diff{}}{\diff{u}} \mtx{A}(u) \right].
\end{equation}
The symbol $\cdot$ refers to ordinary matrix multiplication.  
\end{fact}

\subsection{Proof of Lemma~\ref{lem:trace-moments}}

Let us treat the random matrix $\mtx{X}$ as a matrix-valued function of the
standard normal vector $\vct{\gamma} := (\gamma_1, \dots, \gamma_n)$.
That is,
$$
\mtx{X} = \mtx{X}(\vct{\gamma}) = \sum\nolimits_{i=1}^n \gamma_i \mtx{H}_i.
$$
Write $\mtx{X} = \mtx{X} \cdot \mtx{X}^{2p-1}$ and distribute the sum in the first
factor:
$$
\Expect \trace \mtx{X}^{2p}
	= \Expect \trace\left[ \left(\sum\nolimits_{i=1}^n \gamma_i \mtx{H}_i \right) \mtx{X}^{2p-1} \right]
	= \sum\nolimits_{i=1}^n \Expect\big[ \gamma_i \cdot \trace \big[ \mtx{H}_i \mtx{X}^{2p - 1} \big] \big]
$$
The Gaussian integration by parts formula, Fact~\ref{fact:gauss-ibp}, implies that
$$
\Expect \trace \mtx{X}^{2p}
	= \sum\nolimits_{i=1}^n \Expect \trace\big[ \mtx{H}_i \cdot \partial_i \big(\mtx{X}^{2p-1} \big) \big].
$$
Since $\partial_i \mtx{X} = \mtx{H}_i$,
the derivative formula~\eqref{eqn:d-power} yields
$$
\Expect \trace \mtx{X}^{2p}
	= \sum\nolimits_{i=1}^n \trace\left[ \mtx{H}_i \cdot \sum\nolimits_{q=0}^{2p-2} \mtx{X}^q \mtx{H}_i \mtx{X}^{2p-2-q} \right]
	= \sum\nolimits_{q=0}^{2p-2} \sum\nolimits_{i=1}^n \trace\big[ \mtx{H}_i \mtx{X}^q \mtx{H}_i \mtx{X}^{2p-2-q} \big].
$$
This completes the proof of the formula~\eqref{eqn:gauss-trace-moments}.

\section{A Short Proof of the Matrix Khintchine Inequality}
\label{sec:khintchine}

Historically, proofs of the matrix Khintchine inequality have been rather complicated,
but the result is actually an immediate consequence of Lemma~\ref{lem:trace-moments}.
We will present this argument in detail because it has not appeared in the literature.
Furthermore, the approach serves as a template for the more sophisticated theorems
that are the main contributions of this paper.
Let us restate Proposition~\ref{prop:khintchine} in the form that we will establish it.

\begin{proposition}[Matrix Khintchine] \label{prop:khintchine-body}
Let $\mtx{X} := \sum_{i=1}^n \gamma_i \mtx{H}_i$ be an Hermitian matrix Gaussian series,
as in~\eqref{eqn:matrix-gauss-series}.
Define the matrix variance and standard deviation parameters
\begin{equation} \label{eqn:kbody-var}
\mtx{V} := \mathbf{Var}(\mtx{X}) = \sum\nolimits_{i=1}^n \mtx{H}_i^2
\quad\text{and}\quad
\sigma_{2q} := \left( \trace \mtx{V}^q \right)^{1/(2q)}
\quad\text{for each $q \geq 1$.}
\end{equation}
Then, for each integer $p \geq 1$,
\begin{equation} \label{eqn:kbody-ineq}
\left( \Expect \trace \mtx{X}^{2p} \right)^{1/(2p)}
	\leq \sqrt{2p-1} \cdot \sigma_{2p}.
\end{equation}
\end{proposition}

\noindent
The short proof of Proposition~\ref{prop:khintchine-body} appears in the next two sections.
The approach parallels the exchangeable pairs method
that has been used to establish the matrix Khintchine inequality
for Rademacher series~\cite[Cor.~7.3]{MJCFT14:Matrix-Concentration}.
Here, we replace exchangeable pairs with the conceptually
simpler argument based on Gaussian integration by parts.
To reach the statement of Proposition~\ref{prop:khintchine},
we simply rewrite the trace in terms of a Schatten norm.

\begin{remark}[Noninteger Moments]
Our proof of Proposition~\ref{prop:khintchine-body} can
be adapted to obtain moment bounds for all $p \geq 2$.
See~\cite[Cor.~7.3]{MJCFT14:Matrix-Concentration} for
a closely related argument.
\end{remark}

\subsection{Preliminaries}

The main idea in the proof is to simplify the trace moment identity~\eqref{eqn:gauss-trace-moments}
with an elementary matrix inequality.  Anticipating subsequent arguments, we state
the inequality in greater generality than we need right now.

\begin{proposition} \label{prop:matrix-heinz}
Suppose that $\mtx{H}$ and $\mtx{A}$ are Hermitian matrices of the same size.  Let $q$ and $r$
be integers that satisfy $0 \leq q \leq r$.  For each real number $s$ in the range $0 \leq s \leq \min\{q, r-q\}$,
$$
\trace\big[ \mtx{HA}^q \mtx{HA}^{r-q} \big]
	\leq \trace \big[ \mtx{H} \abs{\mtx{A}}^s \mtx{H} \abs{\mtx{A}}^{r-s} \big].
$$
\end{proposition}

The proof of Proposition~\ref{prop:matrix-heinz} depends on a numerical fact.
For nonnegative numbers $\alpha$ and $\beta$, the function
$\theta \mapsto \alpha^\theta \beta^{1-\theta} + \alpha^{1-\theta} \beta^\theta$
is convex on the interval $[0,1]$, and it achieves its minimum at $\theta = \half$.  Therefore,
\begin{equation} \label{eqn:heinz-ineq}
\alpha^\theta \beta^{1-\theta} + \alpha^{1-\theta} \beta^\theta
	\leq \alpha^{\theta'} \beta^{1-\theta'} + \alpha^{1-\theta'} \beta^{\theta'}
	\quad\text{when $0 \leq \theta' \leq \min\big\{\theta, 1 - \theta \big\}$.}
\end{equation}
We need to lift this scalar inequality to matrices.

\begin{proof} Without loss of generality, we may change coordinates so that $\mtx{A}$ is diagonal:
$\mtx{A} = \sum\nolimits_i a_i \mathbf{E}_{ii}$.
Expanding both copies of $\mtx{A}$,
$$
\trace\big[ \mtx{HA}^q \mtx{HA}^{r-q} \big]
	= \sum\nolimits_{i,j} a_i^q a_j^{r-q} \cdot \trace \big[ \mtx{H} \mathbf{E}_{ii} \mtx{H} \mathbf{E}_{jj} \big]
	= \sum\nolimits_{i,j} \half \big(a_i^q a_j^{r-q} + a_i^{r-q} a_j^{q} \big) \cdot
	\trace \big[ \mtx{H} \mathbf{E}_{ii} \mtx{H} \mathbf{E}_{jj} \big].
$$
After we take absolute values, the inequality~\eqref{eqn:heinz-ineq} implies that
$$
a_i^q a_j^{r-q} + a_i^{r-q} a_j^{q}
	\leq \abs{a_i}^s \abs{\smash{a_j}}^{r-s} + \abs{a_i}^{r-s} \abs{\smash{a_j}}^{s}.
$$
The remaining trace is nonnegative:
$\trace \big[ \mtx{H} \mathbf{E}_{ii} \mtx{H} \mathbf{E}_{jj} \big] = \abssq{ \smash{h_{ij}} }$,
where $h_{ij}$ are the components of the matrix $\mtx{H}$.  As a consequence,
$$
\trace\big[ \mtx{HA}^q \mtx{HA}^{r-q} \big]
	\leq \sum\nolimits_{i,j} \half \big( \abs{a_i}^{s} \abs{\smash{a_j}}^{r-s} + \abs{a_i}^{r-s} \abs{\smash{a_j}}^{s} \big)
	\cdot \trace \big[ \mtx{H} \mathbf{E}_{ii} \mtx{H} \mathbf{E}_{jj} \big]
	= \trace \big[ \mtx{H} \abs{\mtx{A}}^{s} \mtx{H} \abs{\mtx{A}}^{r-s} \big].
$$
To reach the last identity, we reversed our steps to reassemble the sum into a trace.
\end{proof}

\subsection{Proof of the Matrix Khintchine Inequality} \label{sec:pf-matrix-khintchine}

We may now establish Proposition~\ref{prop:khintchine-body}.
Let us introduce notation for the quantity of interest:
$$
E^{2p} := \Expect \trace \mtx{X}^{2p}.
$$
Use the integration by parts result, Lemma~\ref{lem:trace-moments}, to rewrite the trace moment:
$$
E^{2p} = \sum\nolimits_{q=0}^{2p-2} \sum\nolimits_{i=1}^n
	\Expect \trace \big[ \mtx{H}_i \mtx{X}^q \mtx{H}_i \mtx{X}^{2p-2-q} \big].
$$
For each choice of $q$, apply the matrix inequality from Proposition~\ref{prop:matrix-heinz}
with $r = 2p-2$ and $s = 0$ to reach
$$
E^{2p} \leq (2p-1) \sum\nolimits_{i=1}^n \Expect \trace\big[ \mtx{H}_i^2 \mtx{X}^{2(p-1)} \big]
	= (2p - 1) \cdot \Expect \trace\big[ \mtx{V} \mtx{X}^{2(p-1)} \big]
$$
We have identified the matrix variance $\mtx{V}$ defined in~\eqref{eqn:kbody-var}.

Next, let us identify a copy of $E$ on the right-hand side
and solve the resulting algebraic inequality.
To that end, invoke H{\"o}lder's inequality~\eqref{eqn:holder}
for the trace with $\varrho = p$ and $\varrho' = p/(p-1)$:
\begin{align*}
E^{2p} &\leq (2p - 1) \cdot \left( \trace \mtx{V}^p \right)^{1/p} \cdot
	 \Expect \left( \trace \mtx{X}^{2p} \right)^{(p-1)/p} \\
	 &\leq (2p - 1) \cdot \sigma_{2p}^2 \cdot
	 \left( \Expect \trace \mtx{X}^{2p} \right)^{(p-1)/p}
	 = (2p - 1) \cdot \sigma_{2p}^2 \cdot E^{2(p - 1)}.
\end{align*}
We have identified the quantity $\sigma_{2p}$ from~\eqref{eqn:kbody-var}.
The second inequality is Lyapunov's.
Since the unknown $E$ is nonnegative, we can solve the polynomial inequality to reach
$$
E \leq \sqrt{2p-1} \cdot \sigma_{2p}.
$$
This is the required result.

\section{A Second-Order Matrix Khintchine Inequality}
\label{sec:k2}

In this Section, we prove Theorem~\ref{thm:k2-intro}, the
second-order matrix Khintchine inequality.  Let
us restate the result in the form that we will establish it.

\begin{theorem}[Second-Order Matrix Khintchine]  \label{thm:k2-body}
Let $\mtx{X} = \sum_{i=1}^n \gamma_i \mtx{H}_i$ be an Hermitian matrix Gaussian series,
as in~\eqref{eqn:matrix-gauss-series}.  Define the matrix
variance and standard deviation parameter
\begin{equation} \label{eqn:k2-body-variances}
\mtx{V} := \mathbf{Var}(\mtx{X}) = \sum\nolimits_{i=1}^n \mtx{H}_i^2
\quad\text{and}\quad
\sigma_{2p} := \left( \trace \mtx{V}^p \right)^{1/(2p)}
\quad\text{for $p \geq 1$.}
\end{equation}
Define the matrix alignment parameter
\begin{equation} \label{eqn:k2-body-alignment}
w_{2p} := \max_{\mtx{Q}_\ell} \ \left( \trace \abs{
	\sum\nolimits_{i,j=1}^n \mtx{H}_i \mtx{Q}_1 \mtx{H}_j \mtx{Q}_2 \mtx{H}_i \mtx{Q}_3 \mtx{H}_j }^{p/2} \right)^{1/(2p)}
	\quad\text{for $p \geq 1$}
\end{equation}
where the maximum ranges over a triple $(\mtx{Q}_1, \mtx{Q}_2, \mtx{Q}_3)$ of unitary matrices.
Then, for each integer $p \geq 3$,
\begin{equation} \label{eqn:k2-body-ineq}
\left( \Expect \trace \mtx{X}^{2p} \right)^{1/(2p)}
	\leq 3 \sqrt[4]{2p-5} \cdot \sigma_{2p} + \sqrt{2p-4} \cdot w_{2p}.
\end{equation}
\end{theorem}

\noindent
The proof of Theorem~\ref{thm:k2-body} will occupy us for the rest of the section.
To reach the statement in the introduction, we rewrite traces in terms of Schatten
norms. 
We also provide the proof of Proposition~\ref{prop:sigma-versus-w} in Section~\ref{sec:interleave-stdev}.

\subsection{Discussion}
\label{sec:k2-discussion}

Before we establish Theorem~\ref{thm:k2-body}, let us spend a moment to discuss the
proof of this result.
Theorem~\ref{thm:k2-body} is based on the same pattern of argument as the matrix Khintchine inequality,
Proposition~\ref{prop:khintchine-body}.  This time, we apply Proposition~\ref{prop:matrix-heinz} more
surgically to control the terms in the trace moment identity from Lemma~\ref{lem:trace-moments}.
The most significant new observation is that we can use complex interpolation to reorganize the products
of matrices that arise during the calculation.

We can refine this argument in several ways.  First, if we apply complex interpolation with more care,
it is possible to define the matrix alignment parameter~\eqref{eqn:k2-body-alignment} as a 
maximum over the set
$$
\big\{ \mtx{Q}_1, \mtx{Q}_2, \mtx{Q}_3 \text{ are commuting unitaries and $\mtx{Q}_\ell = \Id$ for some $\ell$} \big\}.
$$
Given that commuting matrices are simultaneously diagonalizable, this improvement might make it easier
to bound the matrix alignment parameters.

Second, it is quite clear from the proof that we can proceed beyond the second-order terms.  For example,
for an integer $p \geq 3$, we can obtain results in terms of the third-order quantities
$$
\begin{aligned}
w_{2p,1} &:= \max_{\mtx{Q}_\ell} \ \left( \trace \abs{ \sum\nolimits_{i,j,k=1}^n
	\mtx{H}_i \mtx{Q}_1 \mtx{H}_j \mtx{Q}_2 \mtx{H}_k \mtx{Q}_3 \mtx{H}_i \mtx{Q}_4 \mtx{H}_j \mtx{Q}_5 \mtx{H}_k }^{p/3}
	\right)^{1/(2p)} \\
w_{2p,2} &:= \max_{\mtx{Q}_\ell} \ \left( \trace \abs{ \sum\nolimits_{i,j,k=1}^n
	\mtx{H}_i \mtx{Q}_1 \mtx{H}_j \mtx{Q}_2 \mtx{H}_k \mtx{Q}_3 \mtx{H}_i \mtx{Q}_4 \mtx{H}_k \mtx{Q}_5 \mtx{H}_j }^{p/3}
	\right)^{1/(2p)}.
\end{aligned}
$$
The ordering of indices is $(i,j,k,i,j,k)$ and $(i,j,k,i,k,j)$, respectively.
This refinement allows us reduce the order of coefficient on the standard deviation term $\sigma_{2p}$
in~\eqref{eqn:k2-body-ineq} to $p^{1/6}$.
Unfortunately, we must also compute both alignment parameters $w_{2p,1}$ and $w_{2p,2}$,
instead of just $w_{2p}$.  This observation shows why it is unproductive to press forward
with this approach.  Indeed, the number of orderings of indices grows super-exponentially as we consider
longer products, which is an awful prospect for applications.

\subsection{Preliminaries}

In the proof of Theorem~\ref{thm:k2-body}, we will use two interpolation results
to reorganize products of matrices.  The first one is a
type of matrix H{\"o}lder inequality~\cite[Cor.~1]{LP86:Inegalites-Khintchine}.
Here is a version of the result specialized to our setting.

\begin{fact}[Lust-Piquard] \label{fact:pisier-xu}
Consider a finite sequence $(\mtx{A}_1, \dots, \mtx{A}_n)$ of Hermitian matrices with the same dimension,
and let $\mtx{B}$ be a positive-semidefinite matrix of the same dimension.  For each number $\varrho \geq 2$,
$$
\left( \trace \left( \sum\nolimits_{i=1}^n \mtx{A}_i \mtx{B} \mtx{A}_i \right)^{\varrho/2} \right)^{2/\varrho}
	\leq \left( \trace \left( \sum\nolimits_{i=1}^n \mtx{A}_i^2 \right)^{\varrho} \right)^{1/\varrho}
	\cdot \left( \trace \mtx{B}^{\varrho} \right)^{1/\varrho}.
$$
\end{fact}

\noindent
See~\cite[Lem.~1.1]{PX97:Noncommutative-Martingale} for a proof
based on the Hadamard Three-Lines Theorem~\cite[Prop.~9.1.1]{Gar07:Inequalities-Journey}.

The second result is a more complicated interpolation for a multilinear function
whose arguments are powers of random matrices.

\begin{proposition}[Multilinear Interpolation] \label{prop:interpolation}
Suppose that $F : (\mathbb{M}_d)^k \to \C$ is a multilinear function.
Fix nonnegative integers $\alpha_1, \dots, \alpha_k$ with $\sum_{i=1}^k \alpha_i = \alpha$.
Let $\mtx{Y}_i \in \mathbb{H}_d$ be random matrices, not necessarily independent,
for which $\Expect \norm{ \mtx{Y}_i }^\alpha < \infty$.  Then
$$
\abs{ \Expect F\big( \mtx{Y}_1^{\alpha_1}, \ \dots,\ \mtx{Y}_k^{\alpha_k} \big) }
	\leq \max_{i=1,\dots,k} \Expect
	\max_{\mtx{Q}_{\ell}} \abs{ F\big( \mtx{Q}_1, \ \dots,\ \mtx{Q}_{i-1},\ 
	\mtx{Q}_i \mtx{Y}_i^{\alpha},\ \mtx{Q}_{i+1},\ \dots,\ \mtx{Q}_k \big) }.
$$
In this expression, each $\mtx{Q}_\ell$ is a (random) unitary matrix that commutes with $\mtx{Y}_\ell$.
\end{proposition}

\noindent
As with Fact~\ref{fact:pisier-xu}, the proof of Proposition~\ref{prop:interpolation}
depends on the Hadamard Three-Lines Theorem~\cite[Prop.~9.1.1]{Gar07:Inequalities-Journey}.
The argument is standard but somewhat involved, so we postpone the details to Appendix~\ref{sec:interpolation}.

\subsection{The Overture}

Let us commence with the proof of Theorem~\ref{thm:k2-body}.
The initial steps are similar with the argument that leads
to the matrix Khintchine inequality, Proposition~\ref{prop:khintchine-body}.
Introduce notation for the quantity of interest:
\begin{equation} \label{eqn:k2-trace-moment}
E^{2p} := \Expect \trace \mtx{X}^{2p}
	= \sum\nolimits_{q=0}^{2p-2} \sum\nolimits_{i=1}^n \Expect \trace \big[ \mtx{H}_i \mtx{X}^q \mtx{H}_i \mtx{X}^{2p-2-q} \big].
\end{equation}
The identity follows from the integration by parts result, Lemma~\ref{lem:trace-moments}.

This time, we make finer estimates for the summands in~\eqref{eqn:k2-trace-moment}.
Apply Proposition~\ref{prop:matrix-heinz} with $s = 0$ to the terms where $q \in \big\{0, 1, 2p - 3, 2p - 2 \big\}$.
For the remaining $2p - 5$ values of the exponent $q$, apply Proposition~\ref{prop:matrix-heinz} with $s = 1$.
We reach the bound
\begin{equation} \label{eqn:k2-less-young}
E^{2p} \leq 4 \sum\nolimits_{i=1}^n \Expect \trace \big[ \mtx{H}_i^2 \mtx{X}^{2p-2} \big]
	+ (2p-5) \sum\nolimits_{i=1}^n \Expect \trace \big[ \mtx{H}_i \mtx{X}^{2} \mtx{H}_i \mtx{X}^{2p-4} \big].
\end{equation}
We can take advantage of the fact that the $\mtx{H}_i$ are interleaved with the
powers $\mtx{X}^r$ of the random matrix in the second term.

\subsection{The First Term}

To treat the first term on the right-hand side of~\eqref{eqn:k2-less-young}, simply repeat the arguments
from Section~\ref{sec:pf-matrix-khintchine} to obtain a bound in terms of the quantity $E$.
We have
\begin{equation} \label{eqn:k2-first-term}
\sum\nolimits_{i=1}^n \Expect \trace \big[ \mtx{H}_i^2 \mtx{X}^{2p-2} \big]
	= \Expect \trace \big[ \mtx{V} \mtx{X}^{2p-2} \big]
	\leq \left( \trace \mtx{V}^p \right)^{1/p} \cdot \left( \Expect \trace \mtx{X}^{2p} \right)^{2(p-1)}
	= \sigma_{2p}^2 \cdot E^{2(p-1)}.
\end{equation}
The quantities $\mtx{V}$ and $\sigma_{2p}$ are defined in~\eqref{eqn:k2-body-variances}, and we have identified a copy of $E$.

\subsection{Integration by Parts, Again}

To continue, we want to break down the matrix $\mtx{X}^2$ that appears in the second term
on the right-hand side of~\eqref{eqn:k2-less-young}.  To do so, we perform
another Gaussian integration by parts.
Write $\mtx{X}^2 = \sum_{j=1}^n \gamma_j \mtx{H}_j \mtx{X}$,
and invoke Fact~\ref{fact:gauss-ibp} to obtain
\begin{multline} \label{eqn:k2-ibp2}
\sum\nolimits_{i=1}^n \Expect \trace \big[ \mtx{H}_i \mtx{X}^2 \mtx{H}_i \mtx{X}^{2p-4} \big]
	= \sum\nolimits_{i,j=1}^n \Expect \left[ \gamma_j \cdot \trace \big[ \mtx{H}_i \mtx{H}_j \mtx{X} \mtx{H}_i \mtx{X}^{2p-4} \big] \right] \\ 
	= \sum\nolimits_{i,j=1}^n \Expect \trace\big[ \mtx{H}_i \mtx{H}_j^2 \mtx{H}_i \mtx{X}^{2p-4} \big]
	+ \sum\nolimits_{i,j=1}^n \Expect \trace\left[ \mtx{H}_i \mtx{H}_j \mtx{X} \mtx{H}_i \left( \sum\nolimits_{r=0}^{2p-5} \mtx{X}^r \mtx{H}_j \mtx{X}^{2p-5-r} \right) \right].
\end{multline}
This result follows from the product rule and the formula~\eqref{eqn:d-power} for the derivative of a power.
We will bound the first term on the right-hand side of~\eqref{eqn:k2-ibp2} in terms of the
standard deviation parameter $\sigma_{2p}$, and the second term will lead to the matrix alignment parameter $w_{2p}$.

\subsection{Finding the Standard Deviation Parameter}

Let us address the first term on the right-hand side of~\eqref{eqn:k2-ibp2}.
First, draw the sum back into the trace and identify the matrix variance $\mtx{V}$,
defined in~\eqref{eqn:k2-body-variances}:
$$
\sum\nolimits_{i,j=1}^n \Expect \trace\big[ \mtx{H}_i \mtx{H}_j^2 \mtx{H}_i \mtx{X}^{2p-4} \big]
	= \Expect \trace \left[ \left( \sum\nolimits_{i=1}^n \mtx{H}_i \mtx{V} \mtx{H}_i \right) \mtx{X}^{2(p-2)} \right].
$$
To isolate the random matrix $\mtx{X}$,
apply H{\"o}lder's inequality~\eqref{eqn:holder} with exponents $\varrho = p/2$ and $\varrho' = p/(p-2)$,
and follow up with Lyapunov's inequality.  Thus,
$$
\Expect \trace \left[ \left( \sum\nolimits_{i=1}^n \mtx{H}_i \mtx{V} \mtx{H}_i \right) \mtx{X}^{2(p-2)} \right]
	\leq \left( \trace \left( \sum\nolimits_{i=1}^n \mtx{H}_i \mtx{V} \mtx{H}_i \right)^{p/2} \right)^{2/p}
	\cdot \left(\Expect \trace \mtx{X}^{2p} \right)^{(p-2)/p}.
$$
The Lust-Piquard inequality, Fact~\ref{fact:pisier-xu}, with $\varrho = p$ implies that
$$
\left( \trace \left(\sum\nolimits_{i=1}^n \mtx{H}_i \mtx{V} \mtx{H}_i \right)^{p/2} \right)^{2/p}
	\leq \left( \trace \left(\sum\nolimits_{i=1}^n \mtx{H}_i^2 \right)^{p} \right)^{1/p}
	\cdot \left( \trace \mtx{V}^p \right)^{1/p}
	= \left( \trace \mtx{V}^{p} \right)^{2/p}
	= \sigma_{2p}^4.
	$$
Once again, we identified $\mtx{V}$ and $\sigma_{2p}$ from~\eqref{eqn:k2-body-variances}.
Combine the last three displays to arrive at
\begin{equation} \label{eqn:k2-variance}
\sum\nolimits_{i,j=1}^n \Expect \trace\big[ \mtx{H}_i \mtx{H}_j^2 \mtx{H}_i \mtx{X}^{2p-4} \big]
	\leq \sigma_{2p}^4 \cdot \left(\Expect \trace \mtx{X}^{2p} \right)^{(p-2)/p}
	= \sigma_{2p}^4 \cdot E^{2(p-2)}.
\end{equation}
We have identified another copy of $E$.

\subsection{Finding the Matrix Alignment Parameter}
\label{sec:find-interleave}

It remains to study the second term on the right-hand side of~\eqref{eqn:k2-ibp2}.  Rearranging the sums,
we write this object as
$$
\sum\nolimits_{r=0}^{2p-5} \sum\nolimits_{i,j=1}^n \Expect \trace\left[ \mtx{H}_i \mtx{H}_j \mtx{X} \mtx{H}_i \mtx{X}^r \mtx{H}_j \mtx{X}^{2p-5-r} \right].
$$
We can apply the interpolation result, Proposition~\ref{prop:interpolation}, to consolidate the powers of
the random matrix $\mtx{X}$.  Consider the multilinear function
$$
F(\mtx{A}_1, \mtx{A}_2, \mtx{A}_3)
	:= \sum\nolimits_{i,j=1}^n \trace \big[ \mtx{H}_i \mtx{H}_j \mtx{A}_1 \mtx{H}_i \mtx{A}_2 \mtx{H}_j \mtx{A}_3 \big].
$$
Since $\mtx{X}$ is a matrix Gaussian series, it has moments of all orders.  Therefore, for each index $r$,
\begin{multline*}
\abs{ \Expect \trace\left[ \mtx{H}_i \mtx{H}_j \mtx{X} \mtx{H}_i \mtx{X}^r \mtx{H}_j \mtx{X}^{2p-5-r} \right] } \\
	\leq \max\left\{ \Expect \max_{\mtx{Q}_\ell}
	\abs{ F\big(\mtx{Q}_1 \mtx{X}^{2p-4}, \mtx{Q}_2, \mtx{Q}_3 \big)  }, \
	\Expect \max_{\mtx{Q}_\ell} \abs{ F\big(\mtx{Q}_1, \mtx{Q}_2  \mtx{X}^{2p-4}, \mtx{Q}_3 \big)  }, \
	\Expect \max_{\mtx{Q}_\ell} \abs{ F\big(\mtx{Q}_1, \mtx{Q}_2, \mtx{Q}_3 \mtx{X}^{2p-4} \big)  } \right\}.
\end{multline*}
All three terms in the maximum admit the same bound, so we may as well consider the third one:
$$
\begin{aligned}
\Expect \max_{\mtx{Q}_\ell} \abs{  F(\mtx{Q}_1, \mtx{Q}_2, \mtx{Q}_3 \mtx{X}^{2p-4})  }
	&= \Expect \max_{\mtx{Q}_\ell}  \abs{ \sum\nolimits_{i,j=1}^n \trace \big[ \mtx{H}_i \mtx{H}_j \mtx{Q}_1 \mtx{H}_i
	\mtx{Q}_2 \mtx{H}_j \mtx{Q}_3 \mtx{X}^{2(p-2)}  \big] } \\
	&= \Expect \max_{\mtx{Q}_\ell}  \abs{ \trace \left[ \left( \sum\nolimits_{i,j=1}^n \mtx{Q}_3
	\mtx{H}_i \mtx{H}_j \mtx{Q}_1 \mtx{H}_i \mtx{Q}_2 \mtx{H}_j \right) \mtx{X}^{2(p-2)} \right] } \\
	&\leq \Expect \max_{\mtx{Q}_\ell} \left[ \left( \trace \abs{ \sum\nolimits_{i,j=1}^n
	\mtx{H}_i \mtx{H}_j \mtx{Q}_1 \mtx{H}_i \mtx{Q}_2 \mtx{H}_j }^{p/2} \right)^{2/p} \cdot
	\left( \trace \mtx{X}^{2p} \right)^{(p-2)/p} \right] \\
	&\leq \max_{\mtx{Q}_\ell} \left( \trace \abs{ \sum\nolimits_{i,j=1}^n
	\mtx{H}_i \mtx{H}_j \mtx{Q}_1 \mtx{H}_i \mtx{Q}_2 \mtx{H}_j }^{p/2} \right)^{2/p} \cdot
	\left( \Expect \trace \mtx{X}^{2p} \right)^{(p-2)/p} \\
	&\leq w_{2p}^{4} \cdot E^{2(p-2)}.
\end{aligned}
$$
The first step is the definition of $F$.  To reach the second line, we use the fact that
$\mtx{Q}_3$ commutes with $\mtx{X}$, then we cycle the trace.  The third line is
H{\"o}lder's inequality~\eqref{eqn:holder} with $\varrho = p/2$ and $\varrho' = p/(p-2)$,
and we have used the left unitary invariance of the matrix absolute value to delete $\mtx{Q}_3$.
Next, take the maximum over all unitary matrices, and apply Lyapunov's inequality to
draw the expectation into the term involving $\mtx{X}$.  Finally, identify the quantity $E$
and note that the maximum is bounded by the alignment parameter $w_{2p}^4$, defined in~\eqref{eqn:k2-body-alignment}.
Similar calculations are valid for the other two terms, whence
$$
\abs{ \Expect \trace\left[ \mtx{H}_i \mtx{H}_j \mtx{X} \mtx{H}_i \mtx{X}^r \mtx{H}_j \mtx{X}^{2p-5-r} \right] } \\
	\leq w_{2p}^{4} \cdot E^{2(p-2)}.
$$
Since there are $2p - 4$ possible choices of $r$, we determine that
\begin{equation} \label{eqn:k2-interleave}
\sum\nolimits_{r=0}^{2p-5} \sum\nolimits_{i,j=1}^n \Expect \trace\left[ \mtx{H}_i \mtx{H}_j \mtx{X} \mtx{H}_i \mtx{X}^r \mtx{H}_j \mtx{X}^{2p-5-r} \right]
	\leq (2p-4) \cdot w_{2p}^4 \cdot E^{2(p-2)}.
\end{equation}
The main part of the argument is finished.

\subsection{Putting the Pieces Together}

To conclude, we merge the bounds we have obtained and solve the resulting inequality
for the quantity $E$.
Combine~\eqref{eqn:k2-less-young},~\eqref{eqn:k2-first-term},~\eqref{eqn:k2-ibp2},~\eqref{eqn:k2-variance}, and~\eqref{eqn:k2-interleave} to reach
$$
E^{2p} \leq 4 \sigma_{2p}^2 \cdot E^{2 (p-1)}
	+ (2p-5) \cdot \left[ \sigma_{2p}^4 + (2p-4) \cdot w_{2p}^4 \right]
	\cdot E^{2 (p-2)}.
$$
Clearing factors of $E$, we reach the inequality
$$
E^4 \leq 4 \sigma_{2p}^2 \cdot E^2 + (2p-5) \cdot \left[ \sigma_{2p}^4 + (2p-4) \cdot w_{2p}^4 \right].
$$
If $\alpha$ and $\beta$ are nonnegative numbers,
each nonnegative solution to the quadratic inequality
$t^2 \leq \alpha t + \beta$ must satisfy $t \leq \alpha + \sqrt{\beta}$.
It follows that
$$
E^2 \leq 4 \sigma_{2p}^2 + \sqrt{2p-5} \cdot \left[ \sigma_{2p}^4 + (2p-4) \cdot w_{2p}^4 \right]^{1/2}.
$$
Take the square root, and invoke subadditivity of the square root (twice) to reach
$$
E \leq \big(2 + \sqrt[4]{2p-5} \big) \cdot \sigma_{2p} + \sqrt[4]{(2p-5)(2p-4)} \cdot w_{2p}.
$$
Finally, we simplify the numerical constants to arrive at~\eqref{eqn:k2-body-ineq}.

\subsection{Comparison of Standard Deviation and Alignment Parameters}
\label{sec:interleave-stdev}

Our last task in this section is to establish Proposition~\ref{prop:sigma-versus-w},
which states that the alignment parameter $w_{2p}$ never exceeds the standard deviation $\sigma_{2p}$.
The easiest way to obtain this result is to use block matrices and inequalities for the Schatten norm.

Fix an integer $p \geq 2$, and fix a triple $(\mtx{Q}_1, \mtx{Q}_2, \mtx{Q}_3)$ of unitary matrices.
Consider the quantity
$$
S := \left( \trace \abs{ \sum\nolimits_{i,j=1}^n \mtx{H}_i \mtx{Q}_1 \mtx{H}_j \mtx{Q}_2 \mtx{H}_i \mtx{Q}_3 \mtx{H}_j }^{p/2} \right)^{2/p}.
$$
To establish Proposition~\ref{prop:sigma-versus-w}, it suffices to show that $S \leq \sigma_{2p}^4$.
Using block matrices and converting the trace into a Schatten norm, we can write
$$
S = \pnorm{p/2}{ \begin{bmatrix} \vdots
	\\  \mtx{Q}_2^\adj \mtx{H}_j \mtx{Q}_1^\adj \mtx{H}_i \\ \vdots \end{bmatrix}^\adj
	\begin{bmatrix} \vdots
	\\ \mtx{H}_i \mtx{Q}_3 \mtx{H}_j \\ \vdots \end{bmatrix} }.
$$
The entries of the block column matrices are indexed by pairs $(i, j)$, arranged in lexicographic order. 
Invoke the Cauchy--Schwarz inequality~\eqref{eqn:holder-schatten} for Schatten norms with $\varrho = p/2$:
$$
S^2 \leq \pnorm{p/2}{ \begin{bmatrix} \vdots
	\\ \mtx{Q}_2^\adj \mtx{H}_j \mtx{Q}_1^\adj \mtx{H}_i \\ \vdots \end{bmatrix}^\adj
	\begin{bmatrix} \vdots
	\\ \mtx{Q}_2^\adj \mtx{H}_j \mtx{Q}_1^\adj \mtx{H}_i \\ \vdots \end{bmatrix} } \times
	\pnorm{p/2}{ \begin{bmatrix} \vdots
	\\ \mtx{H}_i \mtx{Q}_3 \mtx{H}_j \\ \vdots \\ \end{bmatrix}^\adj \begin{bmatrix} \vdots
	\\ \mtx{H}_i \mtx{Q}_3 \mtx{H}_j \\ \vdots \\ \end{bmatrix} }.
$$
Write each product of two block matrices as a sum:
$$
S^2 \leq \pnorm{p/2}{ \sum\nolimits_{i,j=1}^n \mtx{H}_i \mtx{Q}_1^\adj \mtx{H}_j^2 \mtx{Q}_1 \mtx{H}_i } \times
	\pnorm{p/2}{ \sum\nolimits_{i,j=1}^n \mtx{H}_j \mtx{Q}_3^\adj \mtx{H}_i^2 \mtx{Q}_3 \mtx{H}_j }.
$$
The two factors have the same form, so it suffices to bound the first one.  Indeed,
$$
\pnorm{p/2}{ \sum\nolimits_{i,j=1}^n \mtx{H}_i \mtx{Q}_1^\adj \mtx{H}_j^2 \mtx{Q}_1 \mtx{H}_i }
	= \pnorm{p/2}{ \sum\nolimits_{i=1}^n \mtx{H}_i \mtx{Q}_1^\adj \mtx{V} \mtx{Q}_1 \mtx{H}_i }
	\leq \pnorm{p}{ \sum\nolimits_{i=1}^n \mtx{H}_i^2 }
	\cdot \pnorm{p}{ \mtx{Q}_1^\adj \mtx{V} \mtx{Q}_1 }
	= \pnorm{p}{ \mtx{V} }^2
	= \sigma_{2p}^{4}.
$$
We have identified the matrix variance $\mtx{V}$, defined in~\eqref{eqn:k2-body-variances},
and then we applied the Lust-Piquard inequality, Fact~\ref{fact:pisier-xu}, with $\varrho = p$.
We identified $\mtx{V}$ again, invoked unitary invariance of the Schatten norm,
and then we recognized the quantity $\sigma_{2p}$ from~\eqref{eqn:k2-body-variances}.
In summary, we have established that
$
S^2 \leq \sigma_{2p}^{8}.
$
This is what we needed to show.

\section{Second-Order Matrix Khintchine under Strong Isotropy}
\label{sec:k2-strong}

In this section, we prove an extension of Theorem~\ref{thm:k2-strong-intro}
that gives both lower and upper bounds for the trace moments of a strongly
isotropic matrix Gaussian series.

\begin{theorem}[Second-Order Matrix Khintchine under Strong Isotropy] \label{thm:sharp-moment}
Let $\mtx{X} := \sum\nolimits_{i=1}^n \gamma_i \mtx{H}_i$ be an Hermitian matrix Gaussian series,
as in~\eqref{eqn:matrix-gauss-series}.  Assume that $\mtx{X}$ has the strong isotropy property
\begin{equation} \label{eqn:sharp-body-isotropy}
\Expect \mtx{X}^p = (\Expect \ntr \mtx{X}^p) \cdot \Id
\quad\text{for each $p = 0, 1, 2, \dots$.}
\end{equation}
Define the matrix standard deviation parameter and matrix alignment parameter
\begin{equation} \label{eqn:sharp-body-stats}
\sigma := \norm{ \sum\nolimits_{i=1}^n \mtx{H}_i^2 }^{1/2}
\quad\text{and}\quad
w := \max_{\mtx{Q}_{\ell}}\ \norm{ \sum\nolimits_{i,j=1}^n \mtx{H}_i \mtx{Q}_1 \mtx{H}_j \mtx{Q}_2 \mtx{H}_i \mtx{Q}_3 
\mtx{H}_j }^{1/4}.
\end{equation}
The maximum ranges over a triple $(\mtx{Q}_1, \mtx{Q}_2, \mtx{Q}_3)$ of unitary matrices.  Then,
for each integer $p \geq 1$,
$$
\Cat_p^{1/(2p)} \cdot \sigma \cdot \big[ 1 \ -\ (pw/\sigma)^4 \big]_+^{1/(2p)}
\quad\leq\quad \left( \Expect \ntr \mtx{X}^{2p} \right)^{1/(2p)}
\quad\leq\quad \Cat_p^{1/(2p)} \cdot \sigma\ +\ 2^{1/4} p^{5/4} \cdot w.
$$
The lower bound also requires that $p^{7/4} w \leq 0.7 \sigma$.
We have written $\Cat_p$ for the $p$th Catalan number,
 the function $[a]_+ := \max\{a,0\}$, and $\ntr$ is the normalized trace.
\end{theorem}

\noindent
The proof of this result appears below, starting in Section~\ref{sec:trace-moments}.
To reach the statement of Theorem~\ref{thm:k2-strong-intro} in the introduction,
we rewrite normalized traces in terms of Schatten norms.  Fact~\ref{fact:catalan} (below)
states that the Catalan numbers satisfy the bound $\Cat_p \leq 4^p$ for each
$p = 1, 2, 3, \dots$, which gives an explicit numerical form for the upper bound.

\subsection{Discussion}
\label{sec:k2-strong-discussion}

Before we establish Theorem~\ref{thm:sharp-moment}, let us comment on the
proof and the meaning of this result.
The most important observation is that the estimate is extremely accurate, at least for some examples.
In particular, for the GOE matrix $\mtx{X}_{\rm goe}$ defined in~\eqref{eqn:intro-goe},
we showed in Section~\ref{sec:gauss-wigner} that the standard deviation parameter
$\sigma \approx 1$ while the alignment parameter
$w \approx d^{-1/4}$.  Therefore, Theorem~\ref{thm:sharp-moment} implies that
$$
\Expect \ntr \mtx{X}_{\rm goe}^{2p} \approx \Cat_{p}
\quad\text{when $p \ll d^{1/7}$.}
$$
This estimate is sufficient to prove that the limiting spectral distribution of
the GOE is the semicircle law.  See~\cite[Sec.~2.3]{Tao12:Topics-Random} for details
about how to derive the law from the trace moments.  Furthermore, Markov's inequality
implies that the norm $\norm{\smash{\mtx{X}_{\rm goe}}} \approx 2$ with high probability.

The proof of Theorem~\ref{thm:sharp-moment} has a lot in common with the arguments
leading up to Proposition~\ref{prop:khintchine-body} and Theorem~\ref{thm:k2-body}.
The main innovation is that we can use the strong isotropy
to imitate a moment identity that would hold in free probability.  This idea allows us to
remove the dependence on $p$ from the standard deviation term.

Although it may seem that the proof requires the matrix $\mtx{X}$ to be a Gaussian
series, there are analogous techniques, based on the theory of exchangeable
pairs~\cite{MJCFT14:Matrix-Concentration}, that allow us to deal with other
types of random matrix series.  This observation has the potential to lead to
universality laws.
It is also clear from the argument that we could prove related results with an approximate
form of strong isotropy.

In addition, it is possible to extend these ideas to a rectangular matrix
Gaussian series $\mtx{Z} := \sum\nolimits_{i=1}^n \gamma_i \mtx{S}_i \in \mathbb{C}^{d_1 \times d_2}$.
In this case, we consider the Hermitian dilation:
$$
\coll{H}(\mtx{Z}) = \begin{bmatrix} \mtx{0} & \mtx{Z} \\ \mtx{Z}^\adj & \mtx{0} \end{bmatrix}.
$$
The correct analog of strong isotropy is that
$$
\Expect \coll{H}(\mtx{Z})^{2p}
	= \begin{bmatrix} (\Expect \ntr (\mtx{ZZ}^\adj)^p ) \cdot \Id & \mtx{0} \\
	\mtx{0} & (\Expect \ntr (\mtx{Z}^\adj \mtx{Z})^p ) \cdot \Id \end{bmatrix}
	\quad\text{for $p = 0, 1, 2, \dots$.}
$$
This observation allows us to obtain sharp bounds for the trace moments of 
rectangular Gaussian matrices.  In this fashion, we can even show that the
limiting spectral density of a sequence of rectangular Gaussian matrices is the Mar{c}enko--Pastur
distribution, provided that the aspect ratio of the sequence is held constant.

Finally, we remark that similar arguments can be applied to obtain algebraic relations
for the Stieltjes transform of the matrix $\mtx{X}$.  This approach may lead more directly
to limit laws for sequences of random matrices with increasing dimension.
See~\cite[Sec.~2.4.1]{AGZ10:Introduction-Random} or~\cite[Sec.~9]{Kem13:Introduction-Random}
for an argument of this species.

\subsection{Preliminaries}

Aside from the results we have collected so far, the proof of Theorem~\ref{thm:sharp-moment} requires
a few additional ingredients.  First, we state some of the basic and well-known properties of the Catalan
numbers.  
\begin{fact}[Catalan Numbers] \label{fact:catalan}
The $p$th Catalan number is defined by the formula
\begin{equation} \label{eqn:cat-defn}
\Cat_p := \frac{1}{p+1} {2p \choose p}
\quad\text{for $p= 0,1,2,\dots$.}
\end{equation}
In particular, $p \mapsto \Cat_p$ is nondecreasing, and $\Cat_p \leq 4^p$ for each $p$.
The Catalan numbers satisfy the recursion
\begin{equation} \label{eqn:cat-recurse}
\Cat_0 = 1
\quad\text{and}\quad
\Cat_{p+1} = \sum\nolimits_{q=0}^p \Cat_q \Cat_{p-q}.
\end{equation}
\end{fact}

The next result is a covariance identity for a product of centered functions of a
Gaussian vector~\cite[Thm.~2.9.1]{NP12:Normal-Approximations}.
It can be regarded as a refinement of the Poincar{\'e} inequality,
which provides a bound for the variance of a centered function of a Gaussian vector.

\begin{fact}[Gaussian Covariance Identity] \label{fact:ou}
Let $\vct{\gamma}, \vct{\gamma}' \in \R^n$ be independent standard normal vectors.
Let $f, g : \R^n \to \C$ be functions whose derivatives are square integrable with respect
to the standard normal measure, and assume that $\Expect f(\vct{\gamma}) = \Expect g( \vct{\gamma} ) = 0$.
Then
$$
\Expect\big[ f(\vct{\gamma}) \cdot g(\vct{\gamma}) \big]
	= \int_0^\infty  \frac{\diff{t}}{\econst^t} \sum\nolimits_{j=1}^n
	\Expect\big[ (\partial_j f)(\vct{\gamma}) \cdot (\partial_j g)(\vct{\gamma}_t) \big]
\quad\text{where}\quad
\vct{\gamma}_t := \econst^{-t} \vct{\gamma} + \sqrt{1 - \econst^{-2t}} \vct{\gamma}'.
$$
The symbol $\partial_j$ refers to differentiation with respect to the $j$th coordinate.
\end{fact}

\noindent
The usual statement of this result involves the Ornstein--Uhlenbeck semigroup,
but we have given a more elementary formulation.

Finally, we need a bound for the solution to a certain type of
polynomial inequality.  This estimate is related to Fujiwara's
inequality~\cite[Sec.~27]{Mar66:Geometry-Polynomials}.  We include
a proof sketch since we could not locate the precise statement in
the literature.

\begin{proposition}[Polynomial Inequalities] \label{prop:poly-ineq}
Consider an integer $k \geq 3$, and fix positive numbers $\alpha$ and $\beta$.
Then
$$
u^k \leq \alpha + \beta u^{k-2} \quad\text{implies}\quad
u \leq \alpha^{1/k} + \beta^{1/2} 
$$
\end{proposition}

\begin{proof}[Proof Sketch]
Consider the polynomial $\phi : u \mapsto u^k - \beta u^{k-2} - \alpha$.
The Descartes Rule of Signs implies that $\phi$ has exactly one positive root, say $u_+$.
Furthermore, $\phi(u) \leq 0$ for a positive number $u$ if and only if $u \leq u_+$ because $\phi(0) < 0$.
By direct calculation, one may verify that $u_{\star} = \alpha^{1/k} + \beta^{1/2}$ satisfies
$\phi(u_{\star}) > 0$, which means that $u_+ < u_{\star}$.
We conclude that $\phi(u) \leq 0$ implies $u \leq u_+ < u_{\star}$.
\end{proof}

\subsection{The Normalized Trace Moments}
\label{sec:trace-moments}

Let us commence the proof of Theorem~\ref{thm:sharp-moment}.
First, we introduce notation for the normalized trace moments of the matrix Gaussian series:
\begin{equation} \label{eqn:trace-mom-def}
\mu_{p} := \Expect \ntr \mtx{X}^{p}
\quad\text{for $p = 0, 1, 2, \dots$.}
\end{equation}
It is clear that $\mu_0 = 1$.  Since $\mtx{X}$ is a symmetric random variable, the odd trace moments
are zero:
$$
\mu_{2p+1} = 0
\quad\text{for $p = 0, 1, 2, \dots$.}
$$
It remains to calculate the even trace moments.

We can obtain the second moment from a simple argument:
\begin{equation} \label{eqn:mu2-computation}
\sum\nolimits_{i=1}^n \mtx{H}_i^2 = \Expect \mtx{X}^2
	= \left( \Expect \ntr \mtx{X}^2 \right) \cdot \Id
	= \mu_2 \cdot \Id.
\end{equation}
The first identity follows from a direct calculation using the definition~\eqref{eqn:matrix-gauss-series}
of the matrix Gaussian series.  The second identity is the strong isotropy hypothesis~\eqref{eqn:sharp-body-isotropy},
and the last relation is the definition of $\mu_2$.  Take the spectral norm of~\eqref{eqn:mu2-computation} to see that
\begin{equation} \label{eqn:mu2-sigma2}
\mu_2 = \sigma^2.
\end{equation}
We have identified the standard deviation parameter, defined in~\eqref{eqn:sharp-body-stats}.

\subsection{Representation of Higher-Order Moments}

The major challenge is to compute the rest of the even moments.  As usual,
the first step is to invoke Gaussian integration by parts.  For each
integer $p \geq 1$, Lemma~\ref{lem:trace-moments} implies that
$$
\mu_{2(p+1)} = \sum\nolimits_{q=0}^{2p} \sum\nolimits_{i=1}^n
\Expect \ntr \big[ \mtx{H}_i \mtx{X}^q \mtx{H}_i \mtx{X}^{2p-q} \big].
$$
We are considering $\mu_{2(p+1)}$ instead of $\mu_{2p}$ because it makes the
argument cleaner.  To analyze this expression, we will examine each index
$q$ separately and subject each one to the same treatment.

Fix an index $0 \leq q \leq 2p$.  First, we center both $\mtx{X}^q$ and $\mtx{X}^{2p-q}$
by adding and subtracting their expectations:
\begin{align*}
\sum\nolimits_{i=1}^n \Expect \ntr \big[ \mtx{H}_i \mtx{X}^q \mtx{H}_i \mtx{X}^{2p-q} \big]
	=& \sum\nolimits_{i=1}^n \ntr \big[ \mtx{H}_i \big(\Expect \mtx{X}^q \big) \mtx{H}_i \big(\Expect \mtx{X}^{2p-q}\big) \big] \\
	+& \sum\nolimits_{i=1}^n \Expect \ntr \big[ \mtx{H}_i \big(\mtx{X}^q - \Expect \mtx{X}^q\big) \mtx{H}_i \big(\mtx{X}^{2p-q} - \Expect \mtx{X}^{2p-q}\big) \big].
\end{align*}
The cross-terms vanish because each one has zero mean.  It is productive to think of the first
sum on the right-hand side as an approximation to the left-hand side, while the second sum is a perturbation.

Let us focus on the first sum on the right-hand side of the last display.  We can use the strong isotropy
hypothesis~\eqref{eqn:sharp-body-isotropy} to simplify this expression:
$$
\begin{aligned}
\sum\nolimits_{i=1}^n \ntr \big[ \mtx{H}_i \big(\Expect \mtx{X}^q \big) \mtx{H}_i \big(\Expect \mtx{X}^{2p-q}\big) \big]
	&= \sum\nolimits_{i=1}^n \ntr \big[ \mtx{H}_i \big( \big(\Expect  \ntr \mtx{X}^q \big) \cdot \Id \big)
	\mtx{H}_i \big( \big(\Expect  \ntr \mtx{X}^{2p-q} \big) \cdot \Id \big) \big] \\
	&= \left(\ntr \sum\nolimits_{i=1}^n \mtx{H}_i^2 \right) \big(\Expect \ntr \mtx{X}^q \big) \big(\Expect \ntr \mtx{X}^{2p-q} \big) \\
	&= \sigma^2 \cdot \mu_q \mu_{2p-q}.
\end{aligned}
$$
The last identity follows from~\eqref{eqn:mu2-computation} and~\eqref{eqn:mu2-sigma2}.
As a side note,
our motivation here is to imitate the moment identity that would hold if
$\mtx{X}$ and $\mtx{H}_i$ were free from each other, in the sense of free probability.

Finally, we combine the last three displays to reach
\begin{equation} \label{eqn:trace-mom-decomp}
\mu_{2(p+1)} = \sigma^2 \cdot \sum\nolimits_{q=0}^{p} \mu_{2q} \mu_{2(p-q)}
	+ \sum\nolimits_{q=1}^{2p-1} \sum\nolimits_{i=1}^n
	\Expect \ntr \big[ \mtx{H}_i \big(\mtx{X}^q - \Expect \mtx{X}^q\big) \mtx{H}_i \big(\mtx{X}^{2p-q} - \Expect \mtx{X}^{2p-q}\big) \big].
\end{equation}
Observe that we have modified the indexing of both sums.  This step
depends on the facts that $\mu_q = 0$ for odd $q$ and that $\mtx{X}^0 = \Id$.

\subsection{The Perturbation Term}
\label{sec:cov-term}

The next step in the argument is to bound the perturbation term in~\eqref{eqn:trace-mom-decomp}
in terms of the alignment parameter $w$, defined in~\eqref{eqn:sharp-body-stats}.
We will use the Gaussian covariance identity, Fact~\ref{fact:ou}.

To that end, let us explain how to write each summand in the perturbation term as a covariance.
Let $\mtx{H}$ be a real, diagonal matrix: $\mtx{H} = \operatorname{diag}( h_1, \dots, h_d )$.
Expanding the normalized trace, using $\alpha, \beta$ for coordinate indices, we find that
$$
\Expect \ntr \big[ \mtx{H} \big(\mtx{X}^q - \Expect \mtx{X}^q\big) \mtx{H} \big(\mtx{X}^{2p-q} - \Expect \mtx{X}^{2p-q}\big) \big]
	= \frac{1}{d} \sum\nolimits_{\alpha, \beta = 1}^d h_{\alpha} h_{\beta} \cdot
	\Expect \big[ \big(\mtx{X}^q - \Expect \mtx{X}^q\big)_{\alpha \beta}
	\big(\mtx{X}^{2p-q} - \Expect \mtx{X}^{2p-q}\big)_{\beta \alpha} \big].
$$
To apply the Gaussian covariance identity to the expectation, we introduce
a parameterized family $\{ \mtx{X}_t : t \geq 0 \}$ of random matrices where
$$
\mtx{X}_t := \sum\nolimits_{k=1}^n \big(\econst^{-t} \gamma_k + \sqrt{1 - \econst^{-2t}} \gamma'_k \big) \cdot \mtx{H}_k
\quad\text{where $\vct{\gamma}'$ is an independent copy of $\vct{\gamma}$.}
$$
Observe that $\mtx{X}$ and $\mtx{X}_t$ have the same distribution, although
they are dependent.  Fact~\ref{fact:ou} and Fact~\ref{fact:matrix-power} deliver 
$$
\begin{aligned}
\Expect \big[ \big(\mtx{X}^q - \Expect \mtx{X}^q\big)_{\alpha \beta} &
	\big(\mtx{X}^{2p-q} - \Expect \mtx{X}^{2p-q}\big)_{\beta \alpha} \big] \\
	&= \int_0^\infty \frac{\diff{t}}{\econst^t} \sum\nolimits_{j=1}^n \Expect \bigg[
	\bigg( \sum\nolimits_{r=0}^{q-1} \mtx{X}^r \mtx{H}_j \mtx{X}^{q-1-r} \bigg)_{\alpha \beta}
	\bigg( \sum\nolimits_{s=0}^{2p-q-1} \mtx{X}_t^s \mtx{H}_j \mtx{X}_t^{2p-q-1-s} \bigg)_{\beta \alpha} \bigg] \\
	&= \sum\nolimits_{r=0}^{q-1} \sum\nolimits_{s=0}^{2p-q-1} 
	\int_0^\infty \frac{\diff{t}}{\econst^t} \sum\nolimits_{j=1}^n \Expect \big[
	\big( \mtx{X}^r \mtx{H}_j \mtx{X}^{q-1-r} \big)_{\alpha \beta}
	\big( \mtx{X}_t^s \mtx{H}_j \mtx{X}_t^{2p-q-1-s} \big)_{\beta \alpha} \big].
\end{aligned}
$$
Combining these formulas and expressing the result in terms of the normalized trace again,
we find that
\begin{multline*}
\Expect \ntr \big[ \mtx{H} \big(\mtx{X}^q - \Expect \mtx{X}^q\big) \mtx{H} \big(\mtx{X}^{2p-q} - \Expect \mtx{X}^{2p-q}\big) \big] \\
	= \sum\nolimits_{r=0}^{q-1} \sum\nolimits_{s=0}^{2p-q-1} \int_0^\infty \frac{\diff{t}}{\econst^t} \sum\nolimits_{j=1}^n
	\Expect \ntr \big[ \mtx{H} \big( \mtx{X}^r \mtx{H}_j \mtx{X}^{q-1-r} \big)
	\mtx{H} \big( \mtx{X}_t^s \mtx{H}_j \mtx{X}_t^{2p-q-1-s} \big) \big].
\end{multline*}
In fact, this expression is valid for any Hermitian matrix $\mtx{H}$ because of the unitary
invariance of the trace.
Summing the last identity over $\mtx{H}= \mtx{H}_i$, we reach
\begin{multline} \label{eqn:perturbn}
\sum\nolimits_{i=1}^n
\Expect \ntr \big[ \mtx{H}_i \big(\mtx{X}^q - \Expect \mtx{X}^q\big) \mtx{H}_i \big(\mtx{X}^{2p-q} - \Expect \mtx{X}^{2p-q}\big) \big] \\
	= \sum\nolimits_{r=0}^{q-1} \sum\nolimits_{s=0}^{2p-q-1} \int_0^\infty \frac{\diff{t}}{\econst^t} \sum\nolimits_{i,j=1}^n
	\Expect \ntr \big[ \mtx{H}_i \mtx{X}^r \mtx{H}_j \mtx{X}^{q-1-r}
	\mtx{H}_i \mtx{X}_t^s \mtx{H}_j \mtx{X}_t^{2p-q-1-s} \big].
\end{multline}
At this point, the alignment parameter $w$ starts to become visible.

\subsection{Finding the Matrix Alignment Parameter}

Our next goal is to control the expression~\eqref{eqn:perturbn} in terms of the alignment
parameter $w$.  To do so, we use the interpolation result, Proposition~\ref{prop:interpolation},
to bound the sum over $(i, j)$.  For each choice of indices $(q,r,s)$, we obtain the estimate
\begin{align*}
\abs{ \sum\nolimits_{i,j =1}^n \Expect \ntr \big[ \mtx{H}_i \mtx{X}^r \mtx{H}_j \mtx{X}^{q-1-r}
	\mtx{H}_i \mtx{X}_t^s \mtx{H}_j \mtx{X}_t^{2p-q-1-s} \big] }
	\leq \max\bigg\{ & \Expect \max_{\mtx{Q}_{\ell}} \abs{ \sum\nolimits_{i,j=1}^n  \ntr \big[
		\mtx{H}_i \mtx{Q}_1 \mtx{X}^{2p-2} \mtx{H}_j \mtx{Q}_2 \mtx{H}_i \mtx{Q}_3 \mtx{H}_j \mtx{Q}_4 \big] }, \\
		& \Expect \max_{\mtx{Q}_{\ell}} \abs{ \sum\nolimits_{i,j=1}^n \ntr \big[
		\mtx{H}_i \mtx{Q}_1 \mtx{H}_j \mtx{Q}_2 \mtx{X}^{2p-2} \mtx{H}_i \mtx{Q}_3 \mtx{H}_j \mtx{Q}_4 \big] }, \\
		& \Expect \max_{\mtx{Q}_{\ell}} \abs{ \sum\nolimits_{i,j=1}^n \ntr \big[
		\mtx{H}_i \mtx{Q}_1 \mtx{H}_j \mtx{Q}_2 \mtx{H}_i \mtx{Q}_3 \mtx{X}_t^{2p-2} \mtx{H}_j \mtx{Q}_4 \big] }, \\
		& \Expect \max_{\mtx{Q}_{\ell}} \abs{ \sum\nolimits_{i,j=1}^n \ntr \big[
		\mtx{H}_i \mtx{Q}_1 \mtx{H}_j \mtx{Q}_2 \mtx{H}_i \mtx{Q}_3 \mtx{H}_j \mtx{Q}_4 \mtx{X}_t^{2p-2} \big] } \bigg\}.
\end{align*}
Each random unitary matrix $\mtx{Q}_\ell$ commutes with the corresponding random matrix $\mtx{X}$ or $\mtx{X}_t$.

As in Section~\ref{sec:find-interleave}, we can bound each term in the maximum in the same fashion.
For example, consider the fourth term:
\begin{align*}
\Expect \max_{\mtx{Q}_{\ell}} \abs{ \sum\nolimits_{i,j=1}^n  \ntr \big[
		\mtx{H}_i \mtx{Q}_1 \mtx{H}_j \mtx{Q}_2 \mtx{H}_i \mtx{Q}_3 \mtx{H}_j \mtx{Q}_4 \mtx{X}_t^{2p-2} \big] }
		&\leq \Expect \max_{\mtx{Q}_{\ell}} \bigg[
		\smnorm{}{ \sum\nolimits_{i,j=1}^n \mtx{H}_i \mtx{Q}_1 \mtx{H}_j \mtx{Q}_2 \mtx{H}_i \mtx{Q}_3 \mtx{H}_j }
		\cdot \big( \ntr \mtx{X}_t^{2p-2} \big) \bigg] \\
		&\leq \max_{\mtx{Q}_{\ell}}
		\smnorm{}{ \sum\nolimits_{i,j=1}^n \mtx{H}_i \mtx{Q}_1 \mtx{H}_j \mtx{Q}_2 \mtx{H}_i \mtx{Q}_3 \mtx{H}_j }
		\cdot \big( \Expect \ntr \mtx{X}_t^{2p-2} \big) \\
		&= w^4 \cdot \mu_{2(p-1)}.
\end{align*}
The first step is H{\"o}lder's inequality for the trace, and the second step is H{\"o}lder's inequality
for the expectation.  In the last line, we recall that $\mtx{X}_t$ has the same distribution as $\mtx{X}$
to identify $\mu_{2(p-1)}$.  Finally, we recognize the matrix alignment parameter $w$, defined
in~\eqref{eqn:sharp-body-stats}.

In summary, we have shown that
$$
\abs{ \sum\nolimits_{i,j =1}^n \Expect \ntr \big[ \mtx{H}_i \mtx{X}^r \mtx{H}_j \mtx{X}^{q-1-r}
	\mtx{H}_i \mtx{X}_t^s \mtx{H}_j \mtx{X}_t^{2p-q-1-s} \big] }
	\leq w^4 \cdot \mu_{2(p-1)}.
$$
Introduce this bound into~\eqref{eqn:perturbn} to arrive at
$$
\abs{ \sum\nolimits_{i=1}^n
\Expect \ntr \big[ \mtx{H}_i \big(\mtx{X}^q - \Expect \mtx{X}^q\big) \mtx{H}_i \big(\mtx{X}^{2p-q} - \Expect \mtx{X}^{2p-q}\big) \big] }
	\leq q(2p-q) \cdot w^4 \cdot \mu_{2(p-1)}
	\leq p^2 \cdot w^4 \cdot \mu_{2(p-1)}.
$$
We have used the numerical inequality $u (a - u) \leq a^2/4$, valid for $u \in \R$.
Finally, we sum this expression over the index $q$ to conclude that
\begin{equation} \label{eqn:covariance-term}
\abs{ \sum\nolimits_{q=0}^{2p-1} \sum\nolimits_{i=1}^n
\Expect \ntr \big[ \mtx{H}_i \big(\mtx{X}^q - \Expect \mtx{X}^q\big) \mtx{H}_i \big(\mtx{X}^{2p-q} - \Expect \mtx{X}^{2p-q}\big) \big] }
	\leq 2p^3 \cdot w^4 \cdot \mu_{2(p-1)}.
\end{equation}
This is the required bound for the perturbation term in~\eqref{eqn:trace-mom-decomp}.

\subsection{A Recursion for the Trace Moments}

In view of~\eqref{eqn:trace-mom-decomp} and~\eqref{eqn:covariance-term}, we have shown that
\begin{equation} \label{eqn:my-recurse-ul}
\mu_{2(p+1)} \quad=\quad \sigma^2 \cdot \sum\nolimits_{q=0}^{p} \mu_{2q} \mu_{2(p-q)}
	\quad\pm\quad 2p^3 \cdot w^4 \cdot \mu_{2(p-1)}
	\quad\text{for $p = 1, 2, 3, \dots$.}
\end{equation}
We have written $\pm$ to indicate that the expression contains both a lower bound and
an upper bound for the normalized trace moment $\mu_{2(p+1)}$.

In the next two sections, we will solve this recursion to obtain explicit bounds on
the trace moments.  First, we obtain the upper bound
\begin{equation} \label{eqn:my-k2-upper}
\mu_{2p}^{1/(2p)} \leq \Cat_p^{1/(2p)} \cdot \sigma  + 2^{1/4} p^{5/4} \cdot w
\quad\text{for $p = 1, 2, 3, \dots$.}
\end{equation}
This result gives us a Khintchine-type inequality.
Afterward, assuming $p^{7/4} w \leq 0.7 \sigma$, we establish the lower bound
\begin{equation} \label{eqn:my-k2-lower}
\mu_{2p}^{1/(2p)} \geq \Cat_p^{1/(2p)} \cdot \sigma \cdot \big[ 1 - (pw/\sigma)^4 \big]_+^{1/(2p)}
\quad\text{for $p = 1, 2, 3, \dots$.}
\end{equation}
Together these estimates yield the statement of Theorem~\ref{thm:sharp-moment}.

\subsection{Solving the Recursion: Upper Bound}

We begin with the proof of the upper bound~\eqref{eqn:my-k2-upper}.
The first step in the argument is to remove the lag term $\mu_{2(p-1)}$
from the recursion~\eqref{eqn:my-recurse-ul} using moment comparison.
Fix an integer $p \geq 1$.  Observe that
$$
\mu_{2(p-1)} = \Expect \ntr \mtx{X}^{2(p-1)}
	\leq \Expect \big( \ntr \mtx{X}^{2(p+1)} \big)^{(p-1)/(p+1)}
	\leq \big( \Expect \ntr \mtx{X}^{2(p+1)} \big)^{(p-1)/(p+1)}
	= \mu_{2(p+1)}^{(p-1)/(p+1)}.
$$
The first inequality holds because $q \mapsto (\ntr \mtx{A}^q)^{1/q}$ is increasing
for any positive-semidefinite matrix $\mtx{A}$, while the second inequality is Lyapunov's.
Introduce this estimate into the recursion~\eqref{eqn:my-recurse-ul} to obtain
$$
\mu_{2(p+1)} \leq \sigma^2 \cdot \sum\nolimits_{q=0}^p \mu_{2q} \mu_{2(p-q)}
	+ 2 p^3 w^4 \cdot \mu_{2(p+1)}^{(p-1)/(p+1)}.
$$
This is a polynomial inequality of the form $u^{p+1} \leq \alpha + \beta u^{p-1}$,
so Proposition~\ref{prop:poly-ineq} ensures that $u \leq \alpha^{1/(p+1)} + \beta^{1/2}$.
In other words,
\begin{equation} \label{eqn:my-recurse-upper}
\mu_{2(p+1)}^{1/(p+1)} \leq \left( \sigma^2 \cdot \sum\nolimits_{q=0}^p \mu_{2q} \mu_{2(p-q)} \right)^{1/(p+1)}
	+ \sqrt{2} \, p^{3/2} \cdot w^2
	\quad\text{for $p = 1, 2, 3, \dots$.}
\end{equation}
Using this formula, we will apply induction to prove that
\begin{equation} \label{eqn:my-recurse-upper-soln}
\mu_{2p}^{1/p} \leq \Cat_p^{1/p} \cdot \sigma^2  + \sqrt{2} \, p^{5/2} \cdot w^2
\quad\text{when $p = 1, 2, 3, \dots$.}
\end{equation}
The stated result~\eqref{eqn:my-k2-upper} follows from~\eqref{eqn:my-recurse-upper-soln}
once we take the square root and invoke subadditivity.

Let us commence the induction.
The formula~\eqref{eqn:my-recurse-upper-soln} holds for $p = 1$ because $\mu_2 = \sigma^2$,
as noted in~\eqref{eqn:mu2-sigma2}.
Assuming that the bound~\eqref{eqn:my-recurse-upper-soln} holds for each integer $p$ in the range
$1, 2, 3, \dots, r$, we will verify that the same bound is also valid for $p = r + 1$.
For any integer $q$ in the range $1 \leq q \leq r$, the bound~\eqref{eqn:my-recurse-upper-soln}
implies that
\begin{equation} \label{eqn:recurse-upper-intermed}
\mu_{2q} \leq  \Cat_q \cdot \sigma^{2q} \cdot \left( 1 + \frac{\sqrt{2} \, q^{5/2} w^2}{\Cat_q^{1/q} \sigma^2} \right)^{q}
	\leq \Cat_q \cdot \sigma^{2q} \cdot \exp\left( \frac{\sqrt{2} \, q^{7/2} w^2}{\sigma^2} \right).
\end{equation}
Using the definition~\eqref{eqn:cat-defn} of the Catalan numbers,
one may verify that $q \mapsto q^{5/2} \Cat_q^{-1/q}$ is increasing,
so
$$
\mu_{2q} \leq  \Cat_q \cdot \sigma^{2q} \cdot \left( 1 + \frac{\sqrt{2}\, (r+1)^{5/2}  w^2}{\Cat_{r+1}^{1/(r+1)} \sigma^2} \right)^{q}
\quad\text{for $q = 0, 1, 2, \dots, r$.}
$$
The case $q = 0$ follows by inspection.
Now, the latter bound and the recursion~\eqref{eqn:cat-recurse} for Catalan numbers together imply that
$$
\sigma^2 \cdot \sum\nolimits_{q=0}^{r} \mu_{2q} \mu_{2(r-q)}
	\leq \Cat_{r+1} \cdot \sigma^{2(r+1)} \cdot \left(1 + \frac{\sqrt{2} \, (r+1)^{5/2}  w^2}{\Cat_{r+1}^{1/(r+1)} \sigma^2} \right)^{r}.
$$
Take the $r + 1$ root to determine that
$$
\begin{aligned}
\left( \sigma^2 \cdot \sum\nolimits_{q=0}^{r} \mu_{2q} \mu_{2(r-q)} \right)^{1/(r+1)}
	&\leq \Cat_{r+1}^{1/(r+1)} \cdot \sigma^2 \cdot \left( 1 + \frac{r}{r+1} \cdot \frac{\sqrt{2} \, (r+1)^{5/2}  w^2}{\Cat_{r+1}^{1/(r+1)} \sigma^2} \right), \\
	&= \Cat_{r+1}^{1/(r+1)} \cdot \sigma^2 + \sqrt{2}\, (r+1)^{5/2} w^2 - \sqrt{2} \, (r+1)^{3/2} w^2. 
\end{aligned}
$$
We have used the numerical inequality $(1+x)^\alpha \leq 1 + \alpha x$, valid for $0 \leq \alpha \leq 1$ and $x \geq 0$.  Combine this estimate with the recursive bound~\eqref{eqn:my-recurse-upper} for $p = r$ to obtain
$$
\mu_{2(r+1)}^{1/(r+1)} \leq \Cat_{r+1}^{1/(r+1)} \cdot \sigma^2 + \sqrt{2}\, (r+1)^{5/2} w^2.
$$
We see that~\eqref{eqn:my-recurse-upper-soln} holds for $p = r + 1$, and the induction may proceed.

\subsection{Solving the Recursion: Lower Bound}

We turn to the proof of the lower bound~\eqref{eqn:my-k2-lower}.
Assuming that $p^{7/4} w \leq 0.7 \sigma$,
we will use induction to show that
\begin{equation} \label{eqn:my-recurse-lower-soln}
\mu_{2p} \geq \sigma^{2p} \cdot \Cat_{p} \cdot \big[1 - (p w/\sigma)^4 \big]_+
\quad\text{for $p = 0, 1, 2, 3, \dots$.}
\end{equation}
The result~\eqref{eqn:my-k2-lower} follows once we take the $(2p)$th root.

To begin the induction, recall that $\mu_0 = 1$, so the formula~\eqref{eqn:my-recurse-lower-soln}
is valid for $p = 0$.  Suppose now that~\eqref{eqn:my-recurse-lower-soln} is valid for
each integer $p$ in the range $0, 1, 2, \dots, r$.  We will verify the formula for
$p = r + 1$.  
The lower branch of the recursion~\eqref{eqn:my-recurse-ul} states that
$$
\mu_{2(r+1)} \geq \sigma^2 \cdot \sum\nolimits_{q=0}^r \mu_{2q} \mu_{2(r-q)}
	- 2 r^3 w^4 \mu_{2(r-1)}.
$$
The induction hypothesis~\eqref{eqn:my-recurse-lower-soln} yields
$$
\begin{aligned}
\sigma^2 \cdot \sum\nolimits_{q=0}^r \mu_{2q} \mu_{2(r-q)}
	&\geq \sigma^{2(p+1)} \cdot \sum\nolimits_{q=0}^r \Cat_q \Cat_{r-q}
	\cdot \big[1 - (qw/\sigma)^4 - ((r-q) w/ \sigma)^4 \big] \\
	&\geq \Cat_{r+1} \cdot \sigma^{2(p+1)} \cdot (rw/\sigma)^4.
\end{aligned}
$$
We have used the fact that $q \mapsto q^4 + (r-q)^4$ achieves its maximum value on $[0, r]$
at one of the endpoints because of convexity.  We also applied the recursive formula~\eqref{eqn:cat-recurse}
for the Catalan numbers.  The bound~\eqref{eqn:recurse-upper-intermed} implies that
$$
\mu_{2(r-1)} \leq 2 \, \Cat_{r-1} \sigma^{2 (r-1) }
\quad\text{when}\quad
\frac{r^{7/4} w}{\sigma} \leq \sqrt{\frac{\log 2}{\sqrt{2}}} \approx 0.7.
$$
Therefore,
$$
2 r^3 w^4 \mu_{2(r-1)}
	\leq 4\, \Cat_{r-1} \cdot \sigma^{2(r-1)} \cdot r^3 w^4 
	\leq 4\,\Cat_{r+1} \cdot \sigma^{2(r+1)} \cdot r^3 (w/\sigma)^4.
$$
The second inequality holds because the Catalan numbers are nondecreasing.
Combine the last three displays to arrive at
$$
\mu_{2(r+1)} \geq \Cat_{r+1} \cdot \sigma^{2(r+1)} \cdot \big[1 - (r^4 + 4r^3)(w/\sigma)^4 \big]
	\geq \Cat_{r+1} \cdot \sigma^{2(r+1)} \cdot \big[1 - (r+1)^4(w/\sigma)^4 \big].
$$
We have verified the formula~\eqref{eqn:my-recurse-lower-soln} for $p = r+1$,
which completes the proof.

\appendix

\section{Interpolation Results}
\label{sec:interpolation}

In this appendix, we establish Proposition~\ref{prop:interpolation}, the
interpolation inequality for a multilinear function of a random matrix,
whose proof appears below in Appendix~\ref{app:interp-multilin}.

\subsection{Multivariate Complex Interpolation}

The interpolation result we use in the body of the paper
is a consequence of a more general theorem on interpolation
for a function of several complex variables.

\begin{proposition}[Multivariate Complex Interpolation] \label{prop:multivar-interp}
Let $k$ be a natural number.  For a positive number $\alpha$, define the simplicial prism
$$
\Delta_k(\alpha) := \left\{ (c_1, \dots, c_k) \in \C^k :
\text{$\real c_i \geq 0$ for each $i$ and $\sum\nolimits_{i=1}^k \real c_i \leq \alpha$} \right\}.
$$
Consider a bounded, continuous function $G : \Delta_k(\alpha) \to \C$.  For each pair
$(i, j)$ of distinct indices and each $\vct{c} \in \Delta_k(\alpha)$, assume that
$G$ has the analytic section property:
\begin{equation} \label{eqn:analytic-section}
z \mapsto G(c_1,\ \dots,\ c_i + z,\ \dots,\ c_j - z,\ \dots,\ c_k)
\quad\text{is analytic on $- \real c_i < \real z < \real c_j$.}
\end{equation}
Then, for each $\vct{z} \in \Delta_k(\alpha)$ with $\beta := \sum_{i=1}^n \real z_i > 0$,
\begin{equation} \label{eqn:multivar-interp}
\abs{ G( z_1, \ \dots,\ z_k ) }
	\leq \bigg( \prod\nolimits_{i=1}^k \ \sup\nolimits_{t_\ell \in \R} \ \abs{ G(\iunit t_1, \ \dots,\ \iunit t_{i-1},\ \beta + \iunit t_i,\ \iunit t_{i+1},\ \dots,\ \iunit t_k) }^{\real z_i} \bigg)^{1/\beta}.
\end{equation}
\end{proposition}

\noindent
We establish Proposition~\ref{prop:multivar-interp} in the next two
sections.  The argument relies on the same principles that support
standard univariate complex interpolation.  Although it seems likely
that a result of this form already appears in the literature, we were
not able to locate a reference.

\subsection{Preliminaries}

Proposition~\ref{prop:multivar-interp} depends
on Hadamard's Three-Lines Theorem~\cite[Prop.~9.1.1]{Gar07:Inequalities-Journey}.

\begin{proposition}[Three-Lines Theorem] \label{prop:three-lines}
Consider the vertical strip $\Delta_1(1)$ in the complex plane:
$$
\Delta_1(1) := \{ z \in \C : 0 \leq \real z \leq 1 \}.
$$
Consider a bounded, continuous function $f : \Delta_1(1) \to \C$ that is
analytic in the interior of $\Delta_1(1)$.  For each $\theta \in [0,1]$,
\begin{equation*} \label{eqn:three-lines}
\sup_{t \in \R} \ \abs{ \smash{f(\theta + \iunit t)} }
	\leq \sup_{t \in \R}\ \abs{ \smash{f(1 + \iunit t)} }^{\theta}
	\cdot \sup_{t \in \R} \ \abs{ \smash{f(\iunit t)} }^{1-\theta}.
\end{equation*}
\end{proposition}

\noindent
As we will see, this result delivers the $k = 2$ case of Proposition~\ref{prop:multivar-interp}.

\subsection{Proof of Proposition~\ref{prop:multivar-interp}}

The proof of the multivariate interpolation result, Proposition~\ref{prop:multivar-interp},
follows by induction on the number $k$ of arguments.

Let us begin with the base cases.  When the function $G$ has one argument only,
the inequality~\eqref{eqn:multivar-interp} is obviously true.  Next, consider a bivariate
function $G_2 : \Delta_2(\alpha) \to \C$ that is bounded and continuous
and has the analytic section property~\eqref{eqn:analytic-section}.
Fix a point $\vct{z} \in \Delta_2(\alpha)$ with $\beta := \real z_1 + \real z_2 > 0$.
Define the bounded, continuous function
$$
f(y) := G_2( \beta y + \iunit \imag z_1, \ \beta (1 - y) + \iunit \imag z_2 )
\quad\text{for $y \in \Delta_1(1)$.}
$$
The assumption~\eqref{eqn:analytic-section} implies that $f$ is analytic on $0 < \real y < 1$.
Select $\theta = \real z_1/\beta$, which gives $1 - \theta = \real z_2/\beta$.
An application of the Three-Lines Theorem, Proposition~\ref{prop:three-lines},
implies that
$$
\abs{ G_2( z_1, \ z_2 ) }
	= \abs{ f(\theta) }
	\leq \sup\nolimits_{t \in \R} \ \abs{ f( 1 + \iunit t ) }^{\theta}
	\cdot \sup\nolimits_{t \in \R} \ \abs{ f( \iunit t ) }^{1-\theta}.
$$
Introducing the definition of $f$ and simplifying,
\begin{align} \label{eqn:interp-case2}	
\abs{ G_2( z_1, \ z_2 ) }
	&\leq \sup\nolimits_{t \in \R} \ \abs{ G_2( \beta ( 1 + \iunit t) + \iunit \imag z_1, \ -\beta \iunit t + \iunit \imag z_2 ) }^{\real z_1 / \beta} \notag \\
	&\qquad\times \sup\nolimits_{t \in \R}\ \abs{ G_2( \beta \iunit t + \iunit \imag z_1, \ \beta ( 1 - \iunit t ) + \iunit \imag z_2) }^{\real z_2 / \beta} \notag \\
	&\leq \bigg( \sup\nolimits_{s_1, s_2 \in \R}\ \abs{ G_2( \beta + \iunit s_1, \ \iunit s_2 ) }^{\real z_1}
	\cdot  \sup\nolimits_{s_1, s_2 \in \R}\ \abs{ G_2( \iunit s_1, \  \beta + \iunit s_2 ) }^{\real z_2}
	\bigg)^{1/\beta}.
\end{align}
This is the $k = 2$ case of Proposition~\ref{prop:multivar-interp}.

Fix a positive integer $k$, and suppose that we have established the inequality~\eqref{eqn:multivar-interp}
for functions with $k - 1$ arguments.  In other words, assume that $G_{k-1} : \Delta_{k-1}(\alpha') \to \C$
is bounded and continuous, and it has the analytic section property~\eqref{eqn:analytic-section}.  Then,
for $\vct{z} \in \Delta_{k-1}(\alpha')$ with $\beta' := \sum_{i=1}^{k-1} \real z_i > 0$,
\begin{equation} \label{eqn:interp-induction}
\abs{ G_{k-1}(z_1, \ \dots,\ z_{k-1}) }
	\leq \bigg( \prod\nolimits_{i=1}^{k-1} \sup\nolimits\nolimits_{t_\ell \in \R} \
	\abs{ G_{k-1}(\iunit t_1,\ \dots,\ \beta' + \iunit t_i,\ \dots,\ \iunit t_{k-1}) }^{\real z_i} \bigg)^{1/\beta'}.
\end{equation}
We need to extend this result to functions with $k$ variables.

Consider a bounded, continuous function $G_k : \Delta_k(\alpha) \to \C$ with the analytic
section property~\eqref{eqn:analytic-section}.  Fix a complex vector
$\vct{z} \in \Delta_k(\alpha)$ with $\beta := \sum\nolimits_{i=1}^k \real z_i > 0$.
Define the number $\beta' := \sum_{i=1}^{k-1} \real z_i$.
When $\beta' = 0$, the formula~\eqref{eqn:multivar-interp}
is trivial for $G = G_k$ because $\real z_1 = \dots = \real z_{k-1} = 0$.
Therefore, we may assume that $\beta' > 0$.
Introduce the function
$$
G_{k-1}(y_1,\ \dots,\ y_{k-1}) := G_k( y_1,\ \dots,\ y_{k-1},\ z_k)
\quad\text{for $\vct{y} \in \Delta_{k-1}(\beta')$.}
$$
One may verify that $G_{k-1}$ inherits boundedness, continuity, and analytic
sections from $G_k$.  Therefore, the induction hypothesis~\eqref{eqn:interp-induction}
gives
\begin{equation} \label{eqn:interp-k-1-vars}
\abs{ G_k(z_1,\ \dots,\ z_{k-1}, \ z_k) }
	\leq \bigg( \prod\nolimits_{i=1}^{k-1} \sup\nolimits_{t_\ell \in \R}\ 
	\abs{ G_k(\iunit t_1, \ \dots, \ \beta' + \iunit t_i, \ \dots,\ \iunit t_{k-1}, \ z_k) }^{\real z_i} \bigg)^{1 / \beta'}. 
\end{equation}
For each fixed choice of the index $i$ and of the numbers $t_1, \dots, t_{k-1} \in \R$,
consider the function
$$
G_2( y_i, \ y_k ) := G_k( \iunit t_1, \ \dots, \ y_i, \ \dots,\ \iunit t_{k-1},\ y_k)
\quad\text{for $(y_i, y_k) \in \Delta_2(\alpha)$.}
$$
Since $\beta' + \real z_k = \beta$, the bivariate case~\eqref{eqn:interp-case2} provides that
\begin{align} \label{eqn:interp-k-vars-partial}
\abs{ G_k( \iunit t_1, \ \dots, \ \beta + \iunit t_i, \ \dots,\  \iunit t_{k-1},\ z_k) }
	\leq\ &\sup\nolimits_{s_i, s_k \in \R} \ \abs{ G_k( \iunit t_1, \ \dots, \ \beta + \iunit s_i, \ \dots,\ \iunit t_{k-1},\ \iunit s_k) }^{\beta'/\beta} \notag \\
	& \times \sup\nolimits_{s_i, s_k \in \R} \ \abs{ G_k( \iunit t_1, \ \dots, \ \iunit s_i, \ \dots,\ \iunit t_{k-1},\ \beta + \iunit s_k) }^{1 - \beta'/\beta}.
\end{align}
Combine the bounds~\eqref{eqn:interp-k-1-vars} and~\eqref{eqn:interp-k-vars-partial} to reach
\begin{align*}
\abs{ G_k(z_1, \ \dots,\ z_{k-1},\ z_k) }
	\leq\ & \bigg( \prod\nolimits_{i=1}^{k-1} \sup\nolimits_{t_\ell \in \R}\ 
	\abs{ G_k( \iunit t_1, \ \dots, \ \beta + \iunit t_i, \ \dots,\ \iunit t_k) }^{\real z_i} \bigg)^{1/\beta}
	 \\
	& \times \bigg( \prod\nolimits_{i=1}^{k-1} \sup\nolimits_{t_\ell \in \R}\ 
	\abs{ G_k( \iunit t_1, \ \dots, \ \iunit t_i, \ \dots,\ \beta + \iunit t_k) }^{\real z_i} \bigg)^{(\beta - \beta')/(\beta' \beta)}
\end{align*}
Since $\beta' = \sum_{i=1}^{k-1} \real z_i$ and $\beta - \beta' = \real z_k$, we see that the second product
has the same form as the $i = k$ term of the first product.  Thus, 
$$
\abs{ G_k(z_1, \ \dots,\ z_k) }
	\leq \bigg( \prod\nolimits_{i=1}^{k} \sup\nolimits_{t_\ell \in \R}\ 
	\abs{ G_k( \iunit t_1, \ \dots, \ \beta + \iunit t_i, \ \dots, \iunit t_k) }^{\real z_i} \bigg)^{1/\beta}.
$$
This step completes the induction.  We have established Proposition~\ref{prop:multivar-interp}.

\subsection{Interpolation for a Multilinear Function of Random Matrices}
\label{app:interp-multilin}

We are now prepared to establish the interpolation result,
Proposition~\ref{prop:interpolation}, for a multilinear
function of random matrices.  We will actually establish
a somewhat more precise version, which we state here.

\begin{proposition}[Refined Multilinear Interpolation] \label{prop:interp-fancy}
Suppose that $F : (\mathbb{M}_d)^k \to \C$ is a multilinear function.
Fix nonnegative integers $\alpha_1, \dots, \alpha_k$ with $\sum_{i=1}^k \alpha_k = \alpha$.
Let $\mtx{Y}_i \in \mathbb{H}_d$ be random Hermitian matrices, not necessarily independent,
for which $\Expect \norm{ \mtx{Y}_i }^\alpha < \infty$.  Then
$$
\abs{ \Expect F\big( \mtx{Y}_1^{\alpha_1}, \ \dots,\ \mtx{Y}_k^{\alpha_k} \big) }
	\leq \left( \prod\nolimits_{i=1}^k  \bigg[ \Expect
	\max_{\mtx{Q}_{\ell}} \abs{ F\big( \mtx{Q}_1, \ \dots,\ \mtx{Q}_{i-1},\ 
	\mtx{Q}_i \mtx{Y}_i^{\alpha},\ \mtx{Q}_{i+1},\ \dots,\ \mtx{Q}_k \big) } \bigg]^{\alpha_i} \right)^{1/\alpha}.
$$
Each $\mtx{Q}_{\ell}$ is a unitary matrix that commutes with $\mtx{Y}_{\ell}$.
\end{proposition}

\noindent
We establish this result in the next section.

Observe that Proposition~\ref{prop:interp-fancy} immediately implies
Proposition~\ref{prop:interpolation}, the interpolation result that
we use in the body of the paper.  Indeed, we recognize the
large parenthesis on the right-hand side as a geometric mean,
and we bound the geometric mean by the maximum of its components.

\subsection{Proof of Proposition~\ref{prop:interp-fancy}}
\label{app:interp-multilin-pf}

By a perturbative argument, we may assume that each matrix $\mtx{Y}_i$
is almost surely nonsingular.  Indeed, for a parameter $\eps > 0$,
we can replace each $\mtx{Y}_i$ by the modified matrix
$\widetilde{\mtx{Y}}_i := \mtx{Y}_i + \eps \gamma_i \Id$,
where $\{\gamma_i\}$ is an independent family of standard normal variables.
After completing the argument, we can draw $\eps$ down to zero to obtain
the inequality for the original random matrices.

The first step in the argument is to perform a polar factorization
of each random Hermitian matrix: $\mtx{Y}_i = \mtx{U}_i \mtx{P}_i$ where
$\mtx{U}_i$ is unitary, $\mtx{P}_i$ is almost surely positive definite,
and the two factors commute with $\mtx{Y}_i$ for each index $i$.
For clarity of argument, we introduce the unitary matrices
$\mtx{S}_i = \mtx{U}_i^{\alpha_i}$.  With this notation,
\begin{equation} \label{eqn:interp-redux}
\abs{ \Expect F\big( \mtx{Y}_1^{\alpha_1}, \ \dots,\ \mtx{Y}_k^{\alpha_k} \big) }
	= \abs{ \Expect F\big( \mtx{S}_1 \mtx{P}_1^{\alpha_1}, \ \dots,\ \mtx{S}_k \mtx{P}_k^{\alpha_k} \big) }
\end{equation}
We will perform interpolation only on the positive-definite matrices.

Next, we introduce a complex-valued function by replacing the
powers $\alpha_i$ with complex variables:
\begin{equation*} \label{eqn:multlin-complex}
G : (z_1, \dots, z_k) \mapsto
	\Expect F\big( \mtx{S}_1 \mtx{P}_1^{z_1}, \ \dots, \ \mtx{S}_k \mtx{P}_k^{z_k} \big)
	\quad\text{for $\vct{z} \in \Delta_k(\alpha)$.}
\end{equation*}
The set $\Delta_k(\alpha)$ is the simplicial prism defined in the statement
of Proposition~\ref{prop:multivar-interp}.

\begin{claim} \label{app:claim}
The function $G$ is bounded, continuous, and has analytic sections.
\end{claim}

\noindent
These are the properties required to apply the interpolation
result, Proposition~\ref{prop:multivar-interp}.

Let us assume that Claim~\ref{app:claim} holds so that we can complete the proof.
The relation~\eqref{eqn:interp-redux} and Proposition~\ref{prop:multivar-interp}
imply that
$$
\begin{aligned}
\abs{ \Expect F( \mtx{Y}_1^{\alpha_1}, \ \dots, \ \mtx{Y}_k^{\alpha_k} ) }
	&= \abs{ \Expect F(\mtx{S}_1 \mtx{P}_1^{\alpha_1}, \ \dots, \ \mtx{S}_k \mtx{P}_k^{\alpha_k}) } \\
	&\leq \bigg( \prod\nolimits_{i=1}^k \sup\nolimits_{t_\ell \in \R} \abs{ \Expect
	F\big( \mtx{S}_1 \mtx{P}_1^{\iunit t_1}, \ \dots,\ \mtx{S}_i \mtx{P}_i^{\alpha + \iunit t_i}, \
	\dots \mtx{S}_k \mtx{P}_k^{\iunit t_k} \big) }^{\alpha_i} \bigg)^{1/\alpha}.
\end{aligned}
$$
%
%
Fix an index $i$ in the product. 
Introduce the unitary matrix
$\mtx{Q}_{ii}(t_i) = \mtx{S}_i \mtx{P}_i^{\iunit t_i} \mtx{U}_i^{-\alpha}$, where $\mtx{U}_i$ is
the polar factor of $\mtx{Y}_i$.  It follows that
$$
\mtx{S}_i \mtx{P}_i^{\alpha + \iunit t_i}
	= \big(\mtx{S}_i \mtx{P}_i^{\iunit t_i} \mtx{U}_i^{-\alpha} \big)\big( \mtx{U}_i^\alpha \mtx{P}_i^\alpha \big)
	= \mtx{Q}_{ii}(t_i) \mtx{Y}_i^\alpha
$$
Similarly, for each $\ell \neq i$, we can define $\mtx{Q}_{i\ell}(t_{\ell}) = \mtx{S}_\ell \mtx{P}_\ell^{\iunit t_\ell}$.
Therefore,
$$
\begin{aligned}
\abs{ \Expect F\big( \mtx{Y}_1^{\alpha_1}, \ \dots,\ \mtx{Y}_k^{\alpha_k} \big) }
	&\leq \left( \prod\nolimits_{i=1}^k \sup_{t_\ell \in \R} \abs{ 
	\Expect F\big( \mtx{Q}_{i1}(t_1), \ \dots,\ \mtx{Q}_{ii}(t_i) \mtx{Y}_i^{\alpha}, \
	\dots, \ \mtx{Q}_{ik}(t_k) \big) }^{\alpha_i} \right)^{1/\alpha} \\
	&\leq \left( \prod\nolimits_{i=1}^k \bigg[ \Expect \max_{\mtx{Q}_{\ell}} \abs{ 
	F\big( \mtx{Q}_1, \ \dots,\ \mtx{Q}_i \mtx{Y}_i^{\alpha}, \
	\dots, \ \mtx{Q}_k \big) } \bigg]^{\alpha_i} \right)^{1/\alpha}.
\end{aligned}
$$
By construction $\mtx{Q}_{i\ell}(t_{\ell})$ commutes with $\mtx{Y}_\ell$ for each index $\ell$.
In the second line, we apply Jensen's inequality.
Then we relax the supremum to include all unitary matrices $\mtx{Q}_{\ell}$
that commute with the corresponding $\mtx{Y}_{\ell}$.  We can replace the supremum
with a maximum since the unitary group is compact and the function $F$ is continuous.
This is what we needed to show.

Finally, we must verify Claim~\ref{app:claim}. Each multilinear function $F : (\mathbb{M}_d)^k \to \C$
is bounded and continuous:
$$
\abs{F(\mtx{A}_1,\ \dots,\ \mtx{A}_k)} \leq \mathrm{Const} \cdot \prod\nolimits_{i=1}^k \norm{\mtx{A}_i}.
$$
Fix a point $\vct{z} \in \Delta_k(\alpha)$, and let $\beta := \sum_{i=1}^k \real z_i$.
Applying this observation to the function $G$,
$$
\begin{aligned}
\abs{G(z_1, \dots, z_k)} &= \abs{ \Expect F( \mtx{S}_1 \mtx{P}_1^{z_1}, \ \dots, \ \mtx{S}_k \mtx{P}_k^{z_k}) }
	\leq \mathrm{Const} \cdot \Expect \bigg[ \prod\nolimits_{i=1}^k \norm{\smash{\mtx{S}_i \mtx{P}_i^{z_i}}} \bigg] \\
	&= \mathrm{Const} \cdot \Expect \bigg[ \prod\nolimits_{i=1}^k \norm{\mtx{Y}_i}^{\real z_i} \bigg]
	\leq \mathrm{Const} \cdot \Expect \left[ \frac{1}{\beta} \sum\nolimits_{i=1}^k \norm{\mtx{Y}_i} \right]^{\beta}.
\end{aligned}
$$
The first estimate follows from Jensen's inequality and the bound on the multilinear function $F$.
The second inequality depends on the unitary invariance of the spectral norm, the identity
$\norm{\mtx{P}^z} = \norm{\mtx{P}}^{\real z}$, and the polar decomposition $\mtx{Y}_i = \mtx{U}_i \mtx{P}_i$.
The last bound is the inequality between the geometric and arithmetic mean.
Since $\Expect \norm{\mtx{Y}_i}^\alpha < \infty$, and $\beta \leq \alpha$,
we conclude that $G$ is bounded.  Since $F$ is continuous,
an application of the Dominated Convergence Theorem shows that $G$ is a continuous function as well.

The proof that $G$ has analytic sections is similar.  Fix a vector $\vct{c} \in \Delta_k(\alpha)$.
Since $F$ is multilinear, it is easy to check that the map
$$
z \mapsto F\big( \mtx{S}_1 \mtx{P}_1^{c_1}, \ \dots, \ \mtx{S}_i \mtx{P}_i^{c_i + z}, \ \dots, \
	\mtx{S}_j \mtx{P}_j^{c_j - z}, \ \dots,\ \mtx{S}_k \mtx{P}_1^{c_k} \big)
	\quad\text{is analytic on $-\real c_i < z < \real c_j$}
$$
for any fixed choice of $\mtx{S}_\ell$ and $\mtx{P}_{\ell}$ and any pair $(i, j)$
of distinct indices.  Together, the Morera Theorem and the Fubini--Tonelli Theorem allow
us to conclude that
$$
z \mapsto \Expect F\big( \mtx{S}_1 \mtx{P}_1^{c_1}, \ \dots, \ \mtx{S}_i \mtx{P}_i^{c_i + z}, \ \dots, \
	\mtx{S}_j \mtx{P}_j^{c_j - z}, \ \dots, \mtx{S}_k \mtx{P}_1^{c_k} \big)
$$
also is analytic.  Therefore, the analytic section property~\eqref{eqn:analytic-section}
is in force.  Claim~\ref{app:claim} is established.

\section*{Acknowledgments}

Afonso Bandeira is responsible for the argument in Section~\ref{sec:spin},
and Ramon van Handel has offered critical comments.  Parts of this research were completed
at Mathematisches Forschungsinstitut Oberwolfach (MFO) and at Instituto Nacional de
Matem{\'a}tica Pura e Aplicada (IMPA) in Rio de Janeiro.  The author gratefully acknowledges
support from ONR award N00014-11-1002, a Sloan Research Fellowship, and the
Gordon \& Betty Moore Foundation.

\bibliographystyle{myalpha}

\begin{thebibliography}{MJC{\etalchar{+}}14}

\bibitem[AGZ10]{AGZ10:Introduction-Random}
G.~W. Anderson, A.~Guionnet, and O.~Zeitouni.
\newblock {\em An introduction to random matrices}, volume 118 of {\em
  Cambridge Studies in Advanced Mathematics}.
\newblock Cambridge University Press, Cambridge, 2010.

\bibitem[BV14]{BV14:Sharp-Nonasymptotic}
A.~Bandeira and R.~van~Handel.
\newblock Sharp nonasymptotic bounds on the norm of random matrices with
  independent entries.
\newblock Available at \url{http://www.arXiv.org/abs/1408.6185}, Aug. 2014.

\bibitem[Bha97]{Bha97:Matrix-Analysis}
R.~Bhatia.
\newblock {\em Matrix analysis}, volume 169 of {\em Graduate Texts in
  Mathematics}.
\newblock Springer-Verlag, New York, 1997.

\bibitem[BRV13]{BRV13:Optimal-Asymptotic}
A.~Bondarenko, D.~Radchenko, and M.~Viazovska.
\newblock Optimal asymptotic bounds for spherical designs.
\newblock {\em Ann. of Math. (2)}, 178(2):443--452, 2013.

\bibitem[BS10]{BS10:Spectral-Analysis}
Z.~Bai and J.~W. Silverstein.
\newblock {\em Spectral analysis of large dimensional random matrices}.
\newblock Springer Series in Statistics. Springer, New York, second edition,
  2010.

\bibitem[Buc01]{Buc01:Operator-Khintchine}
A.~Buchholz.
\newblock Operator {K}hintchine inequality in non-commutative probability.
\newblock {\em Math. Ann.}, 319:1--16, 2001.

\bibitem[CGT12]{CGT12:Masked-Sample}
R.~Y. Chen, A.~Gittens, and J.~A. Tropp.
\newblock The masked sample covariance estimator: an analysis using matrix
  concentration inequalities.
\newblock {\em Inf. Inference}, 1(1):2--20, 2012.

\bibitem[CT14]{CT14:Subadditivity-Matrix}
R.~Y. Chen and J.~A. Tropp.
\newblock Subadditivity of matrix {$\phi$}-entropy and concentration of random
  matrices.
\newblock {\em Electron. J. Probab.}, 19:no. 27, 30, 2014.

\bibitem[Gar07]{Gar07:Inequalities-Journey}
D.~J.~H. Garling.
\newblock {\em Inequalities: a journey into linear analysis}.
\newblock Cambridge University Press, Cambridge, 2007.

\bibitem[Kem13]{Kem13:Introduction-Random}
T.~Kemp.
\newblock Math 247a: Introduction to random matrix theory.
\newblock Available at
  \url{http://www.math.ucsd.edu/~tkemp/247A/247A.Notes.pdf}, Nov. 2013.

\bibitem[LO94]{LO94:Best-Constant}
R.~Lata{\l}a and K.~Oleszkiewicz.
\newblock On the best constant in the {K}hinchin-{K}ahane inequality.
\newblock {\em Studia Math.}, 109(1):101--104, 1994.

\bibitem[LP86]{LP86:Inegalites-Khintchine}
F.~Lust-Piquard.
\newblock In{\'e}galit{\'e}s de {K}hintchine dans {$C_p$ $(1 < p < \infty)$}.
\newblock {\em C. R. Math. Acad. Sci. Paris}, 303(7):289--292, 1986.

\bibitem[LT91]{LT91:Probability-Banach}
M.~Ledoux and M.~Talagrand.
\newblock {\em Probability in Banach Spaces: Isoperimetry and Processes}.
\newblock Springer, Berlin, 1991.

\bibitem[Mar66]{Mar66:Geometry-Polynomials}
M.~Marden.
\newblock {\em Geometry of polynomials}.
\newblock Second edition. Mathematical Surveys, No. 3. American Mathematical
  Society, Providence, R.I., 1966.

\bibitem[MJC{\etalchar{+}}14]{MJCFT14:Matrix-Concentration}
L.~Mackey, M.~I. Jordan, R.~Y. Chen, B.~Farrell, and J.~A. Tropp.
\newblock Matrix concentration inequalities via the method of exchangeable
  pairs.
\newblock {\em Ann. Probab.}, 42(3):906--945, 2014.

\bibitem[NP12]{NP12:Normal-Approximations}
I.~Nourdin and G.~Peccati.
\newblock {\em Normal approximations with {M}alliavin calculus}, volume 192 of
  {\em Cambridge Tracts in Mathematics}.
\newblock Cambridge University Press, Cambridge, 2012.
\newblock From Stein's method to universality.

\bibitem[NS06]{NS06:Lectures-Combinatorics}
A.~Nica and R.~Speicher.
\newblock {\em Lectures on the combinatorics of free probability}, volume 335
  of {\em London Mathematical Society Lecture Note Series}.
\newblock Cambridge University Press, Cambridge, 2006.

\bibitem[Pis98]{Pis98:Non-commutative-Vector}
G.~Pisier.
\newblock Non-commutative vector valued {$L\sb p$}-spaces and completely
  {$p$}-summing maps.
\newblock {\em Ast\'erisque}, (247):vi+131, 1998.

\bibitem[PX97]{PX97:Noncommutative-Martingale}
G.~Pisier and Q.~Xu.
\newblock Non-commutative martingale inequalities.
\newblock {\em Comm. Math. Phys.}, 189(3):667--698, 1997.

\bibitem[Tao12]{Tao12:Topics-Random}
T.~Tao.
\newblock {\em Topics in random matrix theory}, volume 132 of {\em Graduate
  Studies in Mathematics}.
\newblock American Mathematical Society, Providence, RI, 2012.

\bibitem[Tro12]{Tro12:User-Friendly-FOCM}
J.~A. Tropp.
\newblock User-friendly tail bounds for sums of random matrices.
\newblock {\em Found. Comput. Math.}, 12(4):389--434, 2012.

\bibitem[Tro15]{Tro15:Introduction-Matrix}
J.~A. Tropp.
\newblock {\em An Introduction to Matrix Concentration Inequalities}.
\newblock Foundations and Trends in Machine Learning. Now, 2015.
\newblock To appear. Available at \url{http://www.arXiv.org/abs/1501.01571}.

\bibitem[vH15]{VH15:Spectral-Norm}
R.~van~Handel.
\newblock On the spectral norm of inhomogeneous random matrices.
\newblock Available at \url{http://www.arXiv.org/abs/1502.05003}, Feb. 2015.

\bibitem[Ver12]{Ver12:Introduction-Nonasymptotic}
R.~Vershynin.
\newblock Introduction to the non-asymptotic analysis of random matrices.
\newblock In {\em Compressed sensing}, pages 210--268. Cambridge Univ. Press,
  Cambridge, 2012.

\bibitem[VW08]{VW08:Tight-Frames}
R.~Vale and S.~Waldron.
\newblock Tight frames generated by finite nonabelian groups.
\newblock {\em Numer. Algorithms}, 48(1-3):11--27, 2008.

\end{thebibliography}
\newcommand{\etalchar}[1]{$^{#1}$}

\end{document}